\documentclass[11pt,a4paper]{article}

\usepackage[english]{babel}
\usepackage{amsmath}
\usepackage{amssymb}
\usepackage[dvips]{color}
\usepackage{amsfonts}
\usepackage{array}
\usepackage[dvips]{graphicx}
\usepackage{epsfig}
\usepackage{float}
\usepackage{algorithm}
\usepackage{algorithmic}
\usepackage{setspace}
\usepackage{mathdots}
\usepackage{bm}
\usepackage[font=small,labelfont=bf]{caption}
\usepackage{xspace}
\usepackage[table]{xcolor}
\usepackage[bookmarks=true,pdftex,bookmarksopen=true,hidelinks]{hyperref}
\usepackage{changepage}
\usepackage{tikz}

\allowdisplaybreaks
\renewcommand{\d}{\text{d}}
\newcommand{\qq}{\mathbf q}
\newcommand{\x}{\mathbf x}
\newcommand{\Ghat}{\smash{\widehat{G}_E}}
\newcommand{\Gtilde}{\smash{\widetilde{G}_E}}
\newcommand{\C}{\mathbb C}
\newcommand{\N}{\mathbb N}
\newcommand{\R}{\mathbb R}
\newcommand{\mcA}{\mathcal A}
\newcommand{\mcI}{\mathcal I}
\newcommand{\mcL}{\mathcal L}
\newcommand{\mcP}{\mathcal P}
\newcommand{\mcQ}{\mathcal Q}
\newcommand{\mcS}{\mathcal S}
\newcommand{\abs}[1]{\left|#1\right|}
\newcommand{\abso}[1]{|#1|}
\newcommand{\norm}[2]{\left\|#1\right\|_{#2}}
\newcommand{\normV}[1]{\norm{#1}{V}}
\newcommand{\dual}[3]{\langle #1,#2\rangle_{#3}}
\newcommand{\dualV}[2]{\dual{#1}{#2}{V}}
\newcommand{\coeff}[2]{\smash{{\left(#1\right)_{#2}}}}
\newcommand{\PspaceN}[2]{\mathbb P_{#1}^{#2}\left(\C\right)}
\newcommand{\PspaceV}[2]{\mathbb P_{#1}\left(\C;#2\right)}
\newcommand{\ball}[2]{\mathcal B(#1,#2)}
\newcommand{\MNl}{{[M/N]}}
\newcommand{\padegen}[3]{#1_{[#2/#3]}}
\newcommand{\padef}[1]{\padegen{#1}{M}{N}}
\newcommand{\pade}{\padef{\mcS}}
\newcommand{\padek}{\padegen{\mcS}{M_k}{N_k}}
\newcommand{\padebar}{\padef{\overline{\mcS}}}
\newcommand{\num}{\mathcal P_{[M/N]}}
\newcommand{\numbar}{\overline{\mathcal P}_{[M/N]}}
\newcommand{\den}{Q_{[M/N]}}
\newcommand{\denk}{Q_{[M_k/N_k]}}
\newcommand{\denbar}{\overline{Q}_{[M/N]}}
\newcommand{\HMN}{H_\MNl}
\newcommand{\HMNk}{H_{[M_k/N_k]}}
\newcommand{\jfun}[2]{j_{E,\rho}(#1,#2)}
\newcommand{\jfuntilde}[1]{\smash{\widetilde{j}_{E}}(#1)}
\newcommand{\distS}{d_{\Sigma(\mcL)}}

\newcommand{\review}[1]{#1}

\newtheorem{theorem}{Theorem}[section]
\newtheorem{lemma}[theorem]{Lemma}
\newtheorem{corollary}[theorem]{Corollary}
\newtheorem{definition}[theorem]{Definition}
\newtheorem{remark}[theorem]{Remark}

\newenvironment{proof}{\noindent{\bf Proof. \/}\begin{small}\noindent}{\hfill\EndProofMarker\end{small}}
\newcommand{\EndProofMarker}{$\Box$}
\numberwithin{equation}{section}

\newcommand{\UniVie}{\makebox[0pt][l]{\textsuperscript{\hyperlink{UniVie}{$\sharp$}}}}
\newcommand{\CSQI}{\makebox[0pt][l]{\textsuperscript{\hyperlink{CSQI}{$\S$}}}}
\newcommand\blfootnote[1]{%
  \begingroup
  \renewcommand\thefootnote{}\footnote{#1}%
  \addtocounter{footnote}{-1}%
  \endgroup
}
\newcommand{\hmail}[1]{\href{mailto:#1}{\tt #1}}

\begin{document}

\title{Fast Least-Squares Pad\'e approximation of problems with normal operators and meromorphic structure}

\author{Francesca Bonizzoni\UniVie,\hspace{-.02cm}\footnote{F. Bonizzoni acknowledges partial support from the Austrian Science Fund (FWF) through the project F 65, and has been supported by the FWF Firnberg-Program, grant T998.}\hspace{.02cm} Fabio Nobile\CSQI, Ilaria Perugia\UniVie,\hspace{-.02cm}\footnote{I. Perugia has been funded by the Austrian Science Fund (FWF) through the projects F 65 and P 29197-N32, and by the Vienna Science and Technology Fund (WWTF) through the project MA14-006.}\hspace{.02cm} Davide Pradovera\CSQI\phantom{,}\hspace{.02cm}\footnote{D. Pradovera has been funded by the Swiss National Science Foundation (SNF) through project 182236.}}

\date{}

{
\maketitle 
\blfootnote{The Authors would like to acknowledge the kind hospitality of the Erwin Schr\"odinger International Institute for Mathematics and Physics (ESI), where part of this research was developed under the frame of the Thematic Programme {\it Numerical Analysis of Complex PDE Models in the Sciences}.}}
\vspace{-.75cm}
\begin{center}
\hypertarget{UniVie}{\textsuperscript{$\sharp$}{\small\ Faculty of Mathematics, Universit\"at Wien\\ 
Oskar-Morgenstern-Platz 1, 1090 Wien, Austria\\
\hmail{francesca.bonizzoni@univie.ac.at}, \hmail{ilaria.perugia@univie.ac.at}\\[.15cm]
}}
\hypertarget{CSQI}{\textsuperscript{$\S$}{\small\ CSQI -- MATH, \'Ecole Polytechnique F\'ed\'erale de Lausanne\\
Station 8, CH-1015 Lausanne, Switzerland\\
\hmail{fabio.nobile@epfl.ch}, \hmail{davide.pradovera@epfl.ch}\\
}
}
\end{center}

\vspace*{0.5cm}


\begin{abstract}
In this work, we consider the approximation of Hilbert space-valued meromorphic functions that arise as solution maps of parametric PDEs whose operator is the shift of an operator with normal and compact resolvent, e.g. the Helmholtz equation. In this restrictive setting, we propose a simplified version of the Least-Squares Pad\'e approximation technique \review{studied in \cite{Bonizzoni2016} following \cite{Guillaume1998}}. In particular, the estimation of the poles of the target function reduces to a low-dimensional eigenproblem for a Gramian matrix, allowing for a robust and efficient numerical implementation (hence the ``fast'' in the name). Moreover, we prove several theoretical results that improve and extend those in \cite{Bonizzoni2016}, including the exponential decay of the error in the approximation of the poles, and the convergence in measure of the approximant to the target function. The latter result extends the classical one for scalar Pad\'e approximation to our functional framework. We provide numerical results that confirm the improved accuracy of the proposed method with respect to the one introduced in \cite{Bonizzoni2016} for differential operators with normal and compact resolvent.
\end{abstract}

\section{Introduction}
\label{sec:Intruduction}

Parametric PDEs arise in a wide variety of contexts in physics, applied mathematics, and engineering. In most cases, the interest is in the evaluation or approximation of the solution map
\begin{equation}
\label{eq:sol_map_general}
\mcS({\bm\mu}):{\bm\mu}\mapsto\mcA({\bm\mu})^{-1}{\bm f}({\bm\mu})\text{,}
\end{equation}
which associates a (possibly multi-dimensional) parameter to the corresponding solution of a PDE based on the differential operator $\mcA(\cdot)$ and on the data $\bm f(\cdot)$. The parameter ${\bm\mu}$ usually represents a collection of physical or geometric properties, which characterize the underlying complex system, and are allowed to vary within some range of interest.

In many applications, computing the solution of the underlying PDE by some discretization scheme may be very costly even at a single point ${\bm\mu}$ in the parameter domain. Thus, the direct evaluation of the solution map over a large number of parameter values is unfeasible. Within this framework, model order reduction is often applied to obtain a surrogate solution map, with good approximation properties in the whole parameter range of interest. Depending on the existence and on the stability properties of the resolvent operator $\mcA(\cdot)^{-1}$, difficulties may arise in devising a reasonably accurate reduced model, and special techniques may be required, due to the resolvent $\mcA(\cdot)^{-1}$ not existing or being ``nearly unbounded'' at some points in the range of interest, see e.g. \cite{Chen2010, Lassila2012, Veroy2003}.

One particular and common instance of such problems is related to the lack of coercivity of the parametric PDE over a subset of the parameter range of interest. In this paper, we specifically address this situation by considering parametric PDEs for which the operator has a eigenproblem-like structure, i.e. is of the form
\begin{equation}
\label{eq:operator_linear_identity}
\mcA({\bm\mu})=\mcL-z({\bm\mu})\mcI\text{,}
\end{equation}
with $\mcL$ an operator with sufficient regularity (the exact requirements amount to invertibility, and normality and compactness of the resolvent), $\mcI$ the identity operator, and $z({\bm\mu})$ a complex-valued smooth function. Indeed, such operator lacks a bounded resolvent whenever $z({\bm\mu})$ falls into the spectrum of $\mcL$, and is ``nearly unbounded'' for nearby values of the parameter. The problems which may fall within this framework include the Helmholtz, Maxwell, and Schr\"odinger equations with suitable boundary conditions and constraints, to cite just a few.

In this context, rational approximations of the solution map $\mcS({\bm\mu})$ are particularly appealing, as they can potentially capture those critical values of the parameter ${\bm\mu}$ for which the resolvent is not defined \review{\cite{Bonizzoni2016, Guillaume1998, Hetmaniuk2012, Hetmaniuk2013}}. In this paper we focus on the work \cite{Bonizzoni2016}\review{. There, following the Least-Squares Pad\'e approach introduced in \cite{Guillaume1998} for multivariate complex-valued functions,} a general approach is proposed (in particular, without the hypothesis of normality) to build Pad\'e-type rational approximations of Hilbert space-valued monovariate maps. In particular, the construction of the approximant relies on evaluating the target function and its derivatives at a single point in the parameter domain. Such approximation strategies are summarized in Section~\ref{sec:Pade approximation of the solution map}, where their main convergence results are also stated.

In this work, we focus on problems with the particular structure \eqref{eq:sol_map_general}--\eqref{eq:operator_linear_identity}, with $\mcL^{-1}$ normal and compact, and propose a simplified version of the Least-Squares Pad\'e formulation proposed in \cite{Bonizzoni2016}, which can be constructed by a fast and robust algorithm based on progressive orthogonalization techniques. Moreover, our new ``fast'' method leads to approximations that are more accurate than those produced by the Least-Squares Pad\'e method in \cite{Bonizzoni2016}, by better exploiting the eigenproblem-like structure of the solution map.

The particular normal structure and simplified Pad\'e construction allow us to obtain theoretical convergence results (Theorem~\ref{th:fast_convergence} and Corollary~\ref{th:fast_convergence_speed}) that extend those in \cite{Bonizzoni2016}, by relaxing the hypotheses on the approximant parameters and by showing better convergence rates, as attested also by numerical experiments. In addition, within the framework of this paper, we are able to prove exponential convergence rates (Theorem~\ref{th:pole_convergence}) in the approximation of the critical values of the parameters, for which the PDE is ill-posed.

The outline of this work is as follows. In Section~\ref{sec:Preliminary results}, we describe the precise assumptions on $\mcL$ in \eqref{eq:operator_linear_identity}, and investigate their consequences on the solution map $\mcS$. In Section~\ref{sec:Pade approximation of the solution map}, we briefly summarize the rational approximation technique introduced in \cite{Bonizzoni2016}, along with the corresponding convergence result. In Section~\ref{sec:modified algorithm}, we introduce our new ``fast'' Least-Squares Pad\'e approach. %
In Sections~\ref{sec:conv_denominator_fast} and \ref{sec:conv_pade_fast}, we derive several convergence results in approximating the spectrum of $\mcL$ and the solution map, respectively.
In Section~\ref{sec:algo_details}, some techniques to enhance the numerical stability of the method are described. A numerical experiment comparing the approach of \cite{Bonizzoni2016} with the new one is reported in Section~\ref{sec:compare_algorithms}. Lastly, Section~\ref{sec:conclusions} contains some concluding remarks.

\section{Problem setting}
\label{sec:Preliminary results}

Let $\left(V, \dual{\cdot}{\cdot}{V}\right)$ be a separable Hilbert space over $\C$, with induced norm $\normV{\cdot}$. We consider a bijective linear operator $\mcL:D(\mcL)\subset V\to V$ whose domain $D(\mcL)$ is dense in $V$ and whose resolvent $\mcL^{-1}:V\to V$ is compact and normal, i.e. $$\mcL^{-1}\left(\mcL^{-1}\right)^*=\left(\mcL^{-1}\right)^*\mcL^{-1}\quad\text{over }V\text{,}$$ with $\smash{\left(\mcL^{-1}\right)^*}$ denoting the adjoint of $\mcL^{-1}$, namely $\smash{\left(\mcL^{-1}\right)^*}:V\to V$ linear and bounded, such that $$\dualV{\mcL^{-1}v}{w}=\dualV{v}{\left(\mcL^{-1}\right)^*w}\quad\text{for all }v,w\in V\text{.}$$

The spectral theorem for normal compact operators \cite{Blanchard2003, Kubrusly2012} can be applied to $\mcL^{-1}$, leading to the following properties:
\begin{itemize}
\item the spectrum of $\mcL$, which, since $\mcL$ is closed \cite[Proposition 1.15]{Conway1990}, can be characterized as
\begin{equation}
\label{eq:spectrum_set}
\Sigma(\mcL)=\left\{\lambda\in\C : \exists v\in D(\mcL)\setminus\{0\}, \mcL v=\lambda v\right\}\text{,}
\end{equation}
is discrete and does not include 0;
\item whenever $\Sigma(\mcL)$ is not finite (i.e. when $V$ is infinite-dimensional), its only limit point is $\infty$;
\item for all $\lambda\in\Sigma(\mcL)$, the eigenspace associated to $\lambda$, namely
\begin{equation}
\label{eq:eigenspaces}
V_\lambda=\left\{v\in D(\mcL) : \mcL v=\lambda v\right\}\text{,}
\end{equation}
has finite dimension;
\item the eigenspaces $V_\lambda$ and $V_\nu$ are $V$-orthogonal whenever $\lambda,\nu\in\Sigma(\mcL)$, $\lambda\neq\nu$;
\item the family of orthogonal projections onto the eigenspaces, which we denote by $\{P_\lambda\}_{\lambda\in\Sigma(\mcL)}$ (with the same indexing as the eigenspaces), is a resolution of the identity on $V$, \review{i.e., for any $v\in V$,}
\begin{equation}
\label{eq:spectral_basis}
v=\sum_{\lambda\in\Sigma(\mcL)}P_\lambda v\quad\text{with convergence in }V\text{.}
\end{equation}
\end{itemize}

Given an arbitrary $v^\star\in V$ and the (scalar) parameter $z\in\C\setminus\Sigma(\mcL)$, we consider the problem
\begin{equation}
\label{eq:problem_parametric}
\text{find }\mcS(z)\in D(\mcL)\;:\;\left(\mcL-z\mcI\right)\mcS(z)=v^\star
\end{equation}
(with $\mcI:V\to V$ being the identity operator), which defines uniquely the map
\begin{equation}
\label{eq:solution_map}
\mcS:\C\setminus\Sigma(\mcL)\rightarrow V\text{,}\quad z\mapsto\mcS(z)\text{.}
\end{equation}

In particular, for all $\lambda\in\Sigma(\mcL)$, due to \eqref{eq:spectral_basis} and to the fact that the spectral projector $P_\lambda$ commutes\footnote{For all $\lambda\in\Sigma(\mcL)$, the orthogonal projector $P_\lambda$ commutes with the resolvent $\mcL^{-1}$ \cite{Kubrusly2012}, i.e.
\begin{equation}\label{eq:resolvent_commutes}\tag{$*$}
P_\lambda\mcL^{-1}v=\mcL^{-1}P_\lambda v\quad\text{for all }v\in V\text{.}
\end{equation}
For all $w\in D(\mcL)$ it suffices to set $v=\mcL w$ and to apply $\mcL$ to both sides of \eqref{eq:resolvent_commutes} to obtain the desired result.} with $\mcL$ over $D(\mcL)$, \review{we have that}
\begin{equation}
\label{eq:meromorphic_projections}
P_\lambda v^\star=P_\lambda\left(\mcL-z\mcI\right)\mcS(z)=\left(\mcL-z\mcI\right)P_\lambda\mcS(z)=\left(\lambda-z\right)P_\lambda\mcS(z)\text{.}
\end{equation}
Accordingly, the map \eqref{eq:solution_map} can be expressed as
\begin{equation}
\label{eq:meromorphic_decomposition}
\mcS(z)=\sum_{\lambda\in\Sigma(\mcL)}P_\lambda\mcS(z)=\sum_{\lambda\in\Sigma(\mcL)}\frac{P_\lambda v^\star}{\lambda-z}\quad\text{with convergence in }V\text{,}
\end{equation}
and its $V$-norm at $z\in\C\setminus\Sigma(\mcL)$ is bounded by $\normV{v^\star}/\min_{\lambda\in\Sigma(\mcL)}\abs{\lambda-z}$.

From the orthogonal decomposition \eqref{eq:meromorphic_decomposition}, we can deduce that $\mcS$ is meromorphic over $\C$, and that all its poles are simple and belong to the spectrum of $\mcL$. In particular, it is possible to compute the Taylor coefficients of $\mcS$ at $z_0\in\C\setminus\Sigma(\mcL)$, which we denote by
\begin{equation*}
\coeff{\mcS}{0,z_0}=\mcS(z_0)\quad\text{and}\quad\coeff{\mcS}{\alpha,z_0}=\frac{1}{\alpha!}\frac{\textup{d}^\alpha\mcS}{\textup{d}z^\alpha}(z_0)\quad\text{for }\alpha=1,2,\ldots\text{,}
\end{equation*}
by solving the problems
\begin{equation}
\label{eq:helmholtz_taylor1_parametric_weak}
\text{find }\coeff{\mcS}{\alpha,z_0}\in V\;:\;\left(\mcL-z_0\mcI\right)\coeff{\mcS}{\alpha,z_0}=\coeff{\mcS}{\alpha-1,z_0}\quad\text{for }\alpha=1,2,\ldots\text{.}
\end{equation}

\subsection{Example: the Helmholtz solution map}
\label{sec:helmholtz_solmap}

As an instance of the framework described above, we consider the solution map of the Helmholtz problem with parametric wavenumber and homogeneous Dirichlet boundary conditions, which has been considered by the authors also in \cite{Bonizzoni2016, Bonizzoni2018}.

Let $\Omega\subset\R^d$, $d\in\{1,2,3\}$, be an open Lipschitz bounded domain.
Given $z\in\C$, we consider the Helmholtz problem
\begin{equation}
\label{eq:helmholtz_parametric_weak_complex}
\begin{cases}
-\Delta\mcS(z)-z\mcS(z)=f\quad&\text{in }\Omega\\
\mcS(z)=0\quad&\text{on }\partial\Omega\text{,}
\end{cases}
\end{equation}
with $f\in L^2(\Omega)$. In particular, we cast the problem in the same functional setting as \eqref{eq:problem_parametric}, as
\begin{equation*}
\text{find }\mcS(z)\in D(\Delta)\;:\;\left(-\Delta-z\mcI\right)\mcS(z)=f\;\text{in }L^2(\Omega)\text{,}
\end{equation*}
where we have defined $D(\Delta)=\left\{v\in H^1_0(\Omega) : \Delta v\in L^2(\Omega)\right\}$ and $V=L^2(\Omega)$.

Standard results in functional analysis \cite{Gilbarg1977} can be used to prove that, with the choice of spaces introduced above, $\mcL=-\Delta$ satisfies the hypotheses of the previous section. In particular, it is bijective, and has self-adjoint (hence normal) and compact resolvent. Thus, $\mcS$ is meromorphic and has the form \eqref{eq:meromorphic_decomposition}, with $\Sigma(\mcL)\subset\R^+$ due to the positiveness of $\mcL$.

From \eqref{eq:eigenspaces}, it can be observed that all eigenspaces $V_\lambda$, for $\lambda\in\Sigma(\mcL)$, are subsets of $H^1_0(\Omega)$. Actually, as remarked also in \cite{Bonizzoni2016}, they are mutually orthogonal with respect to the $H^1_0(\Omega)$ inner product as well, and their direct sum in the topology of $H^1_0(\Omega)$ is dense in $H^1_0(\Omega)$. Thus, the spectral expansion \eqref{eq:meromorphic_decomposition} holds true also in $H^1_0(\Omega)$.

\section{Least-Squares Pad\'e approximation}
\label{sec:Pade approximation of the solution map}

In the setting of the previous section, it is reasonable to look for rational approximations of the map $\mcS(z)$. The following Least-Squares (LS) Pad\'e approximant of $\mcS$ centered at $z_0\in \C\setminus\Sigma(\mcL)$ was defined in \cite{Bonizzoni2016, Bonizzoni2018}.

\begin{definition}
    \label{def:pade_approximant}
    Let $z_0\in \C\setminus\Sigma(\mcL)$, $\rho\in\R$ and $N,M,E\in\N$ be fixed, with $\rho>0$ and $E\geq M+N$. Define the polynomial spaces
    \begin{equation}
    \label{eq:num_space}
    \PspaceV{M}{V}=\Bigg\{\mcP:\C\to V,\ \mcP(z)=\sum_{j=0}^Mp_j(z-z_0)^j\text{ with }\{p_j\}_{j=0}^M\subset V\Bigg\}
    \end{equation}
    and
    \begin{equation}
    \label{eq:den_space}
    \PspaceN{N}{z_0}=\Bigg\{Q:\C\to\C,\ Q(z)=\sum_{j=0}^Nq_j(z-z_0)^j\text{ with }\{q_j\}_{j=0}^N\subset\C,\sum_{j=0}^N\left|q_j\right|^2=1\Bigg\}\text{.}
    \end{equation}

    A $\MNl$ LS-Pad\'e approximant of $\mcS$ centered at $z_0$ (which depends parametrically on $E$ and $\rho$) is defined as
    \begin{equation}
    \label{eq:pade}
    \padebar(z)=\frac{\numbar(z)}{\denbar(z)}\text{,}
    \end{equation}
    with $\smash{(\numbar,\denbar)}\in\PspaceV{M}{V}\times\PspaceN{N}{z_0}$ a global minimizer of the functional $j_{E,\rho}:\PspaceV{M}{V}\times\PspaceN{N}{z_0}\to\R^+$, given by
    \begin{equation}
    \label{eq:functional_j}
    \jfun{\mcP}{Q}=\left(\sum_{\alpha=0}^E\normV{\coeff{Q\mcS-\mcP}{\alpha,z_0}}^2\rho^{2\alpha}\right)^{1/2}\text{.}
    \end{equation}
\end{definition}

The minimization of $j_{E,\rho}$ always admits at least one solution, i.e. a $\MNl$ LS-Pad\'e approximant always exists. Indeed, since $\mcP
\in\PspaceV{M}{V}$ has degree \review{at most $M$,}
\begin{equation*}
\jfun{\mcP}{Q}^2=\sum_{\alpha=0}^M\normV{\coeff{Q\mcS-\mcP}{\alpha,z_0}}^2\rho^{2\alpha}+\sum_{\alpha=M+1}^E\normV{\coeff{Q\mcS}{\alpha,z_0}}^2\rho^{2\alpha}\text{.}
\end{equation*}
In particular, for any fixed $Q\in\PspaceN{N}{z_0}$, a (unique) minimizer of
\begin{equation*}
\sum_{\alpha=0}^M\normV{\coeff{Q\mcS-\mcP}{\alpha,z_0}}^2\rho^{2\alpha}
\end{equation*}
over $\PspaceV{M}{V}$, which achieves the value 0, can always be found by imposing the conditions
\begin{equation}
\label{eq:optimal_numerator_standard}
\coeff{\mcP}{\alpha,z_0}=\coeff{Q\mcS}{\alpha,z_0}\quad\text{for }\alpha=0,\ldots,M\text{.}
\end{equation}

Thus, the minimization of $j_{E,\rho}$ can be split into two parts: first, the optimal denominator is computed by minimizing 
\begin{equation}
\label{eq:optimal_denominator_standard}
\overline{j}_{E,\rho}(Q)^2=\sum_{\alpha=M+1}^E\normV{\coeff{Q\mcS}{\alpha,z_0}}^2\rho^{2\alpha}
\end{equation}
over $\PspaceN{N}{z_0}$; a minimizer always exists since $\overline{j}_{E,\rho}(Q)$ is continuous and $\PspaceN{N}{z_0}$ is compact. Then the corresponding optimal numerator is found by enforcing \eqref{eq:optimal_numerator_standard}.

In~\cite{Bonizzoni2016}, the convergence of LS-Pad\'e approximants to the solution map $\mcS$, as $M$ increases while $N$ stays constant, was proven. We recall the result for completeness.

\begin{theorem}{\bf\cite[Theorem 2.4]{Bonizzoni2018}}
	\label{th:pade_approx_S}
	Let $N\in\N$ be fixed. Consider $R>0$ such that the disk $\ball{z_0}{R}=\{z\in\C,\abs{z-z_0}<R\}$ contains at most $N$ poles of $\mcS$, with no element of $\Sigma(\mcL)$ on the boundary $\partial\ball{z_0}{R}$.
	
	Given $\rho<R$, denote by $\padebar$ the $\MNl$ LS-Pad\'e approximant of $\mcS$ at $z_0$ with parameters $E$ and $\rho$. Then, for all $z\in\ball{z_0}{\rho}\setminus\Sigma(\mcL)$ there exists $M^\star\in\N$ such that
	\begin{equation}
	\label{eq:pade_approx_S}
	\normV{\mcS(z)-\padebar(z)}\leq C\left(\frac{\rho}{R}\right)^{M}\quad\text{for all }M\geq M^\star\text{,}
	\end{equation}
	where $C$ depends on $z_0$, $\rho$, $R$, $E-M$, $N$, $\Sigma(\mcL)$, $\min_{\lambda\in\Sigma(\mcL)}\abs{z-\lambda}$, and $\normV{v^\star}$.
\end{theorem}

Several numerical experiments \cite{Bonizzoni2016, Bonizzoni2018} in the case of Helmholtz frequency response problems, lead to believe that the bound \eqref{eq:pade_approx_S} may not be sharp with respect to $\rho$. Actually, no appreciable dependence of the error on $\rho$ has been detected, and the empirically observed convergence rate in $M$ for fixed $N$ was
\begin{equation}
\label{eq:pade_approx_S_empirical}
\normV{\mcS(z)-\padebar(z)}\leq C'\left(\frac{|z-z_0|}{R}\right)^M\text{,}
\end{equation}
see~\cite[Remark 7.1]{Bonizzoni2016} and \cite[Section 4.2]{Bonizzoni2018}, even when $z\in\ball{z_0}{R}\setminus\ball{z_0}{\rho}$, a case which is not addressed by Theorem~\ref{th:pade_approx_S}.

\section{Fast LS-Pad\'e approximants}
\label{sec:modified algorithm}

As the dependence on $\rho$ of the approximation error appears empirically negligible, we may wish to derive a simplified version of LS-Pad\'e approximant that does not depend on $\rho$. Accordingly, we consider the following definition.

\begin{definition}
    \label{def:pade_simplified}
    Let $z_0\in\C\setminus\Sigma(\mcL)$, and $N,M,E\in\N$ be fixed, so that $E\geq\max\{M,N\}$. We define a $\MNl$ fast LS-Pad\'e approximant of $\mcS$ centered at $z_0$ (which depends parametrically on $E$) as
    \begin{equation}
    \label{eq:pade_simplified}
    \pade(z)=\frac{\num(z)}{\den(z)}\text{,}
    \end{equation}
    with $\den\in\PspaceN{N}{z_0}$ being a global minimizer of the functional $\smash{\widetilde{j}_{E}}:\PspaceN{N}{z_0}\to\R^+$, given by
    \begin{equation}
    \label{eq:functional_j_tilde}
    \jfuntilde{Q}=\normV{\coeff{Q\mcS}{E,z_0}}\text{,}
    \end{equation}
    and $\num\in\PspaceV{M}{V}$ satisfying
    \begin{equation*}
    \coeff{\num}{\alpha,z_0}=\coeff{\den\mcS}{\alpha,z_0}\quad\text{for }\alpha=0,\ldots,M\text{.}
    \end{equation*}
\end{definition}

Fast LS-Pad\'e approximants can be formally interpreted as the limit for large $\rho$ of standard LS-Pad\'e approximants given in Definition~\ref{def:pade_approximant}, since the simplified functional $\widetilde{j}_E$ in \eqref{eq:functional_j_tilde} (and, consequently, its minimizer) can be obtained from \eqref{eq:optimal_denominator_standard} by letting $\rho$ tend to $\infty$. To understand what this simplification entails, it is useful to interpret the vectors of coefficients of the denominators $\den$ and $\denbar$ as eigenvectors of Hermitian matrices, as follows.

Any element $Q\in\PspaceN{N}{z_0}$ is uniquely identified by the vector $\qq=(q_j)_{j=0}^N\in\C^{N+1}$ of its components with respect to the basis $\smash{\big((\,\cdot-z_0)^{N-j}\big)_{j=0}^N}$, so that
\begin{equation*}
Q(z)=\sum_{j=0}^N\coeff{Q}{N-j,z_0}(z-z_0)^{N-j}=\sum_{j=0}^Nq_j(z-z_0)^{N-j}\text{.}
\end{equation*}
In particular, as $Q$ is normalized, $\qq$ belongs to the unit sphere of $\C^{N+1}$.

Using this representation, we can express
\begin{align}
\widetilde{j}_E(Q)^2&=\normV{\coeff{Q\mcS}{E,z_0}}^2=\normV{\sum_{j=0}^N\coeff{Q}{N-j,z_0}\coeff{\mcS}{E-N+j,z_0}}^2\nonumber\\
&=\dualV{\sum_{j=0}^Nq_j\coeff{\mcS}{E-N+j,z_0}}{\sum_{i=0}^Nq_i\coeff{\mcS}{E-N+i,z_0}}\nonumber\\
&=\sum_{i=0}^N\sum_{j=0}^N\dualV{\coeff{\mcS}{E-N+j,z_0}}{\coeff{\mcS}{E-N+i,z_0}}\,q_jq_i^*\nonumber\\
&=\qq^*\Gtilde\qq\text{,}\label{eq:simple_functional}
\end{align}
where the unary operator $\textsuperscript{*}$ denotes complex conjugation for scalars and Hermitian transposition for vectors and matrices. In particular, we have defined $\Gtilde\in\C^{(N+1)\times(N+1)}$ as the Hermitian positive semidefinite Gramian matrix whose entries are given by
\begin{equation}
\label{eq:simple_gram}
\left(\Gtilde\right)_{i,j}=\dualV{\coeff{\mcS}{E-N+j,z_0}}{\coeff{\mcS}{E-N+i,z_0}}\quad\text{for }i,j=0,\ldots,N\text{.}
\end{equation}

From equation \eqref{eq:simple_functional}, we infer that a minimizer of $\widetilde{j}_E(Q)$ is a (normalized) eigenvector of $\Gtilde$ corresponding to the minimal eigenvalue. This allows us to compute fast LS-Pad\'e approximants using Algorithm~\ref{algo:fast} below. In practice, the matrix $\Gtilde$ need not be built explicitly, and a better conditioned eigenproblem can be solved instead, as detailed in Section~\ref{sec:algo_details}.

\begin{algorithm}
    \caption{Computation of fast LS-Pad\'e approximants}
    \label{algo:fast}
    \begin{spacing}{1}
    \begin{algorithmic}[1]
        \STATE Fix $z_0\in\C\setminus\Sigma(\mcL)$, $M,N,E\in\N$, with $E\geq\max\{M,N\}$;
        \STATE Compute the Taylor coefficients $\coeff{\mcS}{\alpha,z_0}$ for $\alpha=0,\ldots,E$, by solving \eqref{eq:problem_parametric} and \eqref{eq:helmholtz_taylor1_parametric_weak};
        \STATE Build the Hermitian positive semidefinite Gramian matrix $\Gtilde$ as in \eqref{eq:simple_gram};
        \STATE Compute a normalized eigenvector $\qq\in\C^{N+1}$ of $\Gtilde$ corresponding to the minimal eigenvalue;
        \STATE Define the Pad\'e denominator as $\den=\sum_{j=0}^Nq_j(\cdot-z_0)^{N-j}$;
        \STATE Compute the Taylor coefficients $\coeff{\den\mcS}{\alpha,z_0}$ for $\alpha=0,\ldots,M$;
        \STATE Compute the numerator $\num:=\sum_{\alpha=0}^M\coeff{\den\mcS}{\alpha,z_0}(\,\cdot-z_0)^\alpha$;
        \STATE Return $\pade=\num/\den$.
    \end{algorithmic}
    \end{spacing}
\end{algorithm}

\begin{remark}\label{re:standard_functional}
A similar derivation can be carried out for $\overline{j}_{E,\rho}$ in \eqref{eq:optimal_denominator_standard}, see \cite[Proposition~3.2]{Bonizzoni2018}. In particular, \eqref{eq:simple_functional} becomes
\begin{equation}
\label{eq:standard_functional}
\overline{j}_{E,\rho}(Q)^2=\qq^*\left(\sum_{\gamma=M+1}^E\rho^{2\gamma}\widetilde{G}_\gamma\right)\qq\text{.}
\end{equation}
\end{remark}

In \eqref{eq:standard_functional}, each of the $(N+1)\times(N+1)$ matrices $\widetilde{G}_\gamma$, for $\gamma\geq N$, can be obtained as a diagonal block of the infinite-dimensional Gramian matrix based on the derivatives of $\mcS$, whose entries are defined as
\begin{equation}
\label{eq:full_gram}
\left(G\right)_{i,j}=\dualV{\coeff{\mcS}{j,z_0}}{\coeff{\mcS}{i,z_0}}\quad\text{for }i,j\in\N\text{,}
\end{equation}
see Figure~\ref{fig:gram_matrix}. The matrices $\widetilde{G}_\gamma$ for $\gamma<N$ can be obtained similarly, by adding zero-padding to $G$, or equivalently by defining $\coeff{\mcS}{\alpha,z_0}=0$ for $\alpha<0$.

\begin{figure}[ht]
    \begin{gather*}
    \tikz[remember picture]{\node(c){\hspace{9cm}\Large$\widetilde{G}_3$};}\\[-.75cm]
    \tikz[remember picture,overlay]{
        \draw[-stealth,thick] (c.south)++(4.15cm,.35cm) -- ++(-.9cm,-.4cm);
    }
    G=\left[
    \setlength\arrayrulewidth{0pt}
    \begin{tabular}{cccccc}
    $\dualV{\mcS_0}{\mcS_0}$ & $\dualV{\mcS_1}{\mcS_0}$ &
    $\dualV{\mcS_2}{\mcS_0}$ & $\ldots$ & &\\
    $\dualV{\mcS_0}{\mcS_1}$ & 	
    \cellcolor{blue!25}$\dualV{\mcS_1}{\mcS_1}$ &
    \cellcolor{blue!25}$\dualV{\mcS_2}{\mcS_1}$ &
    \cellcolor{blue!25}$\dualV{\mcS_3}{\mcS_1}$ & 		
    $\ldots$\\
    $\dualV{\mcS_0}{\mcS_2}$ & 		
    \cellcolor{blue!25}$\dualV{\mcS_1}{\mcS_2}$ &
    \cellcolor{blue!25}$\dualV{\mcS_2}{\mcS_2}$ &
    \cellcolor{blue!25}$\dualV{\mcS_3}{\mcS_2}$ & 		
    $\dualV{\mcS_4}{\mcS_2}$ &
    $\ldots$\\
    $\smash\vdots$ & 		
    \cellcolor{blue!25}$\dualV{\mcS_1}{\mcS_3}$ &
    \cellcolor{blue!25}$\dualV{\mcS_2}{\mcS_3}$ &
    \cellcolor{blue!25}$\dualV{\mcS_3}{\mcS_3}$ & 		
    $\dualV{\mcS_4}{\mcS_3}$ &
    $\ldots$\\
    & $\smash\vdots$ &
    $\dualV{\mcS_2}{\mcS_4}$ &
    $\dualV{\mcS_3}{\mcS_4}$ &
    $\dualV{\mcS_4}{\mcS_4}$ &
    $\ldots$\\
    && $\smash\vdots$ & $\smash\vdots$ & $\smash\vdots$ &
    \end{tabular}
    \right]
    \end{gather*}
    \caption{Gramian matrix associated to the map $\mcS$ through the scalar product $\dualV{\cdot}{\cdot}$. To lighten the notation, we write $\mcS_\alpha$ instead of $\coeff{\mcS}{\alpha,z_0}$ (for $\alpha\in\N$) to denote a Taylor coefficient of $\mcS$ at $z_0$. In blue the sub-matrix extracted for $N=2$ and $E=3$, which corresponds to $\smash{\widetilde{G}_3}$.}
    \label{fig:gram_matrix}
\end{figure}

Within this framework, the computation of standard LS-Pad\'e approximants relies on a combination of Gramian blocks, see \eqref{eq:standard_functional}, while, for the same value of $E$, fast approximants only consider the last of these blocks, i.e. the one on the bottom-right.

In the next Section, we derive some properties of Pad\'e denominators, by exploiting features of the Gramian matrix $G$. In particular, we show that diagonal blocks which are related to derivatives of higher order lead to a more accurate estimation of the poles of $\mcS$. As such, in choosing the parameters for standard LS-Pad\'e approximants, we may want to opt for larger values of $\rho$, in order to enhance the contribution of high-order derivatives of $\mcS$. Therefore, fast Pad\'e denominators provide a better approximation of the poles of $\mcS$ than standard Pad\'e ones.

\section{Convergence of fast LS-Pad\'{e} denominators}
\label{sec:conv_denominator_fast}

From here onwards, we will assume without loss of generality that all removable singularities of $\mcS$ have been discarded, i.e. that $v^\star$ is such that $P_\lambda v^\star\neq 0$ for all $\lambda\in\Sigma(\mcL)$. This is not a limiting assumption, since from \eqref{eq:meromorphic_decomposition} it is clear that the poles of $\mcS $ are $\{\lambda\in\Sigma(\mcL):P_\lambda v^\star\neq 0\}$, so that we are entitled to ignore those elements $\lambda\in\Sigma(\mcL)$ for which $P_\lambda v^\star=0$. 

Moreover, we denote by $\{\lambda_\alpha\}_{\alpha=1}^\infty$ the elements of $\Sigma(\mcL)$, ordered in such a way that $|\lambda_\alpha-z_0|\leq|\lambda_{\alpha+1}-z_0|$ for $\alpha=1,2,\ldots$, and we set $v^\star_\alpha=P_{\lambda_\alpha}v^\star$ for $\alpha=1,2,\ldots$. Additionally, we assume that $z_0\in\C\setminus\Sigma(\mcL)$ is fixed and that $\Sigma(\mcL)$ consists of infinitely many elements, unless otherwise explicitly stated (this is just to simplify the notation, since all the results below apply to the finite-dimensional case as well).

\review{In Theorem~\ref{th:pole_convergence} below, we prove that, for a fixed denominator degree $N>0$, the poles of the fast LS-Pad\'e approximant with denominator $\den$ (see Definition~\ref{def:pade_simplified}) converge to the closest poles of $\mcS$, as the number of employed derivatives $E$ goes to $\infty$. More precisely, denoting by $\{\widetilde{\lambda}_\beta^{(E)}\}_{\beta=1}^{N}$ the roots of $\den$, we prove that, for $\alpha=1,\ldots,N$,
\begin{equation*}
    \min_{\beta=1,\ldots,N}\left|\widetilde{\lambda}_\beta^{(E)}-\lambda_\alpha\right|\lesssim\left|\frac{\lambda_\alpha-z_0}{\lambda_{N+1}-z_0}\right|^{2E}\text{,} 
\end{equation*}
where the hidden constant depends on $\alpha$ but is independent of $E$. In order to do that, after rewriting the target functional $\smash{\widetilde{j}_E}$ in Definition~\ref{def:pade_simplified} in a convenient way, we prove three preliminary results in Lemma~\ref{th:interp_bounds}, Lemma~\ref{th:fast_optimality}, and Lemma~\ref{th:den_at_pole_convergence}. We conclude this section by proving convergence of the poles of the fast LS-Pad\'e approximant to the closest poles of $\mcS$ also for increasing $N$, see Theorem~\ref{th:pole_convergence_N} below.}

\review{We start by deriving a useful} alternative expression for $\smash{\widetilde{j}_E}$ in Definition~\ref{def:pade_simplified}. Thanks to \eqref{eq:meromorphic_decomposition}, we can compute each Taylor coefficient of $\mcS$ at $z_0$ as
\begin{equation}
\label{eq:meromorphic_decomposition_taylor}
\coeff{\mcS}{\gamma,z_0}=\sum_{\alpha=1}^\infty\coeff{(\lambda_\alpha-\cdot)^{-1}}{\gamma,z_0}v^\star_\alpha=\sum_{\alpha=1}^\infty\frac{v^\star_\alpha}{(\lambda_\alpha-z_0)^{\gamma+1}}\text{,}
\end{equation}
so that we can express $\jfuntilde{Q}$, for $Q\in\PspaceN{N}{z_0}$, as
\begin{align}
	\widetilde{j}_E(Q)^2&=\normV{\sum_{j=0}^N\coeff{Q}{N-j,z_0}\coeff{\mcS}{E-N+j,z_0}}^2\nonumber\\
	&=\normV{\sum_{\alpha=1}^\infty\sum_{j=0}^N\coeff{Q}{N-j,z_0}(\lambda_\alpha-z_0)^{N-j}\frac{v^\star_\alpha}{(\lambda_\alpha-z_0)^{E+1}}}^2\nonumber\\
	&=\normV{\sum_{\alpha=1}^\infty\frac{v^\star_\alpha}{(\lambda_\alpha-z_0)^{E+1}}Q(\lambda_\alpha)}^2\nonumber\\
    &=\sum_{\alpha=1}^\infty\frac{\normV{v^\star_\alpha}^2}{\left|\lambda_\alpha-z_0\right|^{2E+2}}\left|Q(\lambda_\alpha)\right|^2\text{,}\label{eq:quadratic_simplified_poly}
\end{align}
by the $V$-orthogonality of $\{v^\star_\alpha\}_{\alpha=1}^\infty$.

\review{The first technical lemma} provides some bounds on normalized polynomials in terms of their roots.

\begin{lemma}
    \label{th:interp_bounds}
    Let $Q\in\PspaceN{N}{z_0}$ have (possibly non-distinct) roots $z_1,\ldots,z_N$. For any $z\in\C$ \review{we have the lower bound:}
    \begin{equation}
    \label{eq:interp_bound1}
    \left|Q(z)\right|\geq\prod_{\alpha=1}^N\frac{\left|z_\alpha-z\right|}{1+\left|z_\alpha-z_0\right|}\text{.}
    \end{equation}
    Moreover, if $z_0\notin\{z_\alpha\}_{\alpha=1}^N$, \review{the following upper bound holds true for all $z\in\C$:}
    \begin{equation}
    \label{eq:interp_bound2}
    \left|Q(z)\right|\leq\prod_{\alpha=1}^N\left|\frac{z_\alpha-z}{z_\alpha-z_0}\right|\text{.}
    \end{equation}
\end{lemma}

\begin{proof}
	We can express $Q$ as a normalized interpolation polynomial: there exists $\tau>0$ such that
    \begin{equation}
    \label{eq:exact_den_interp}
    \left|Q(z)\right|=\tau\left|\ell_N(z)\right|\text{,}
    \end{equation}
    where
    \begin{equation}
    \label{eq:exact_den_lag_interp}
    \ell_N(z)=\prod_{\alpha=1}^N\left(z_\alpha-z\right)\text{.}
    \end{equation}
    
    Due to the normalization of $Q$, \review{we have that}
    \begin{equation*}
    \tau^{-2}=\tau^{-2}\sum_{j=0}^N\left|\coeff{Q}{j,z_0}\right|^2=\sum_{j=0}^N\left|\coeff{\ell_N}{j,z_0}\right|^2\text{,}
    \end{equation*}
	which can be evaluated using the Hadamard multiplication theorem \cite[Section 4.6]{Titchmarsh1978}:
    \begin{equation}
    \label{eq:exact_den_scaling}
    \tau^{-2}=\int_0^1\abs{\ell_N\left(z_0+e^{2\pi i\theta}\right)}^2\d\theta=\int_0^1\prod_{\alpha=1}^N\abs{z_\alpha-z_0-e^{2\pi i\theta}}^2\d\theta\text{.}
    \end{equation}
    The two claims will follow from \eqref{eq:exact_den_interp} and \eqref{eq:exact_den_lag_interp} by employing an upper and a lower bound for $\tau^{-2}$, respectively:
	\begin{itemize}
	\item the triangular inequality yields
    \begin{equation*}
    \tau^{-2}\leq\int_0^1\prod_{\alpha=1}^N\left(\abs{z_\alpha-z_0}+\abs{e^{2\pi i\theta}}\right)^2\d\theta=\prod_{\alpha=1}^N\left(\abs{z_\alpha-z_0}+1\right)^2\text{,}
    \end{equation*}
    from which \eqref{eq:interp_bound1} follows;
	\item the Cauchy-Schwarz inequality in $L^2(0,1)$ applied to \eqref{eq:exact_den_scaling} allows to derive
    \begin{align*}
    \tau^{-2}\geq&\abs{\int_0^1\prod_{\alpha=1}^N\left(z_\alpha-z_0-e^{2\pi i\theta}\right)\d\theta}^2\\
    =&\abs{\int_0^1\left(\prod_{\alpha=1}^N\left(z_\alpha-z_0\right)+\sum_{j=1}^Nc_je^{2\pi ij\theta}\right)\d\theta}^2\text{,}
    \end{align*}
    for some coefficients $\{c_j\}_{j=1}^N\subset\C$ independent of $\theta$, whose exact expression is not relevant; indeed, by linearity, \review{it can be shown that}
    \begin{equation*}
	\tau^{-2}\geq\abs{\prod_{\alpha=1}^N\left(z_\alpha-z_0\right)+\sum_{j=1}^Nc_j\int_0^1e^{2\pi ij\theta}\d\theta}^2=\abs{\prod_{\alpha=1}^N\left(z_\alpha-z_0\right)}^2\text{,}
	\end{equation*}
    leading to \eqref{eq:interp_bound2}.
	\end{itemize}	    
\end{proof}

\begin{remark}
    \label{re:interp_continuity}
	In the proof of Lemma~\ref{th:interp_bounds}, it can be observed that both $\tau$ and the absolute value of the interpolation polynomial $\abs{\ell_N(z)}$ at any point $z\in\C$ depend continuously on the roots of $Q$, see \eqref{eq:exact_den_scaling} and \eqref{eq:exact_den_lag_interp}, respectively. Thus, due to \eqref{eq:exact_den_interp}, $\abs{Q(z)}$ depends continuously on the roots of $Q$ as well, for all $z\in\C$.
\end{remark}

The second lemma establishes a sort of optimality bound for fast Pad\'e denominators.

\begin{lemma}
    \label{th:fast_optimality}
    Let $\smash{\widetilde{j}_{E}}$ and $\den$ be the target functional and the fast Pad\'{e} denominator, respectively, as in Definition~\ref{def:pade_simplified}. \review{Then,}
    \begin{equation}
    \label{eq:fast_optimality}
    \widetilde{j}_E\left(\den\right)\leq\frac{C'}{|\lambda_{N+1}-z_0|^{E+1}}\text{,}
    \end{equation}
    with
    \begin{equation}
    \label{eq:fast_optimality_const}
    C'=\normV{v^\star}\prod_{\alpha=1}^N\left(1+\abs{\frac{\lambda_{N+1}-z_0}{\lambda_\alpha-z_0}}\right)\text{.}
    \end{equation}
\end{lemma}

\begin{proof}
	Let $g$ be the exact denominator of $\mcS$ with degree $N$, i.e. some element of $\PspaceN{N}{z_0}$ with roots $\{\lambda_\alpha\}_{\alpha=1}^N$.

    Thanks to \eqref{eq:quadratic_simplified_poly} and to the optimality of $\den$, see Definition~\ref{def:pade_simplified}, \review{we obtain}
    \begin{align*}
    \widetilde{j}_E\left(\den\right)^2\leq\widetilde{j}_E\left(g\right)^2=&\sum_{\alpha=1}^\infty\frac{\normV{v^\star_\alpha}^2}{\left|\lambda_\alpha-z_0\right|^{2E+2}}\left|g(\lambda_\alpha)\right|^2\\
    =&\sum_{\alpha=N+1}^\infty\frac{\normV{v^\star_\alpha}^2}{\left|\lambda_\alpha-z_0\right|^{2E+2}}\left|g(\lambda_\alpha)\right|^2\text{.}
    \end{align*}
	Now Lemma~\ref{th:interp_bounds} and the triangular inequality can be applied, yielding
    \begin{align*}
    \widetilde{j}_E\left(\den\right)^2\leq&\sum_{\alpha=N+1}^\infty\frac{\normV{v^\star_\alpha}^2}{\left|\lambda_\alpha-z_0\right|^{2E+2}}\prod_{\beta=1}^N\abs{\frac{\lambda_\beta-\lambda_\alpha}{\lambda_\beta-z_0}}^2\\
    \leq&\sum_{\alpha=N+1}^\infty\frac{\normV{v^\star_\alpha}^2}{\left|\lambda_\alpha-z_0\right|^{2E+2}}\prod_{\beta=1}^N\left(1+\abs{\frac{\lambda_\alpha-z_0}{\lambda_\beta-z_0}}\right)^2\\
    \leq&\sup_{\alpha\geq N+1}\left(\frac{1}{\left|\lambda_\alpha-z_0\right|^{E+1}}\prod_{\beta=1}^N\left(1+\abs{\frac{\lambda_\alpha-z_0}{\lambda_\beta-z_0}}\right)\right)^2\sum_{\alpha=N+1}^\infty\normV{v^\star_\alpha}^2\text{.}
    \end{align*}

	Since $E\geq N$, the supremum is achieved for $\alpha=N+1$, leading to 
    \begin{equation*}
    \widetilde{j}_E\left(\den\right)^2\leq\frac{1}{|\lambda_{N+1}-z_0|^{2E+2}}\prod_{\beta=1}^N\left(1+\abs{\frac{\lambda_{N+1}-z_0}{\lambda_\beta-z_0}}\right)^2\sum_{\alpha=N+1}^\infty\normV{v^\star_\alpha}^2\text{.}
    \end{equation*}
    
    The claim follows by exploiting the $V$-orthogonality of the $\{v^\star_\alpha\}_{\alpha=1}^\infty$:
    \begin{equation*}
    \sum_{\alpha=N+1}^\infty\normV{v^\star_\alpha}^2\leq\sum_{\alpha=1}^\infty\normV{v^\star_\alpha}^2=\normV{\sum_{\alpha=1}^\infty v^\star_\alpha}^2=\normV{v^\star}^2\text{.}
    \end{equation*}
\end{proof}

\review{The last technical} result provides a bound for the absolute value of the fast LS-Pad\'e denominator when evaluated at the elements of $\Sigma(\mcL)$ closest to $z_0$.

\begin{lemma}
    \label{th:den_at_pole_convergence}
	Let $N\in\N\setminus\{0\}$ be fixed, and consider the fast LS-Pad\'e denominator $\den$ computed with $E\geq N$ derivatives of $\mcS$ at $z_0\in\C\setminus\Sigma(\mcL)$ (the choice of $M$ is irrelevant, as it does not affect $\smash{\widetilde{j}_E}$). Then, for $\alpha=1,\ldots,N$, \review{the fast Pad\'e denominator satisfies the bound:}
    \begin{equation}\label{eq:den_at_pole_convergence}
    \abs{\den(\lambda_\alpha)}\leq c_\alpha\left|\frac{\lambda_\alpha-z_0}{\lambda_{N+1}-z_0}\right|^{2E}\text{,}
    \end{equation}
    with $c_\alpha$ independent of $E$.
\end{lemma}

\begin{proof}
	Let $E\geq N$ be fixed, and consider the vector $\qq_E\in\C^{N+1}$, with $\norm{\qq_E}{2}=1$, such that $\qq_E=\smash{\big(\coeff{\den}{N-j,z_0}\big)_{j=0}^{N}}$. For each $\alpha=1,2,\ldots$, let $\bm{\omega}_\alpha\in\C^{N+1}$ be defined as
	\begin{equation*}
	\bm{\omega}_\alpha=\Big[\left(\lambda_\alpha-z_0\right)^{N},\ldots,\lambda_\alpha-z_0,1\Big]^*\text{,}
	\end{equation*}
	so that $\den(\lambda_\alpha)=\bm{\omega}_\alpha^*\qq_E$.
	
	Moreover, consider the Hermitian matrices $\Gtilde,\Ghat\in\C^{(N+1)\times(N+1)}$ defined as
	\begin{equation*}
	\Gtilde=\sum_{\alpha=1}^\infty\frac{\normV{v^\star_\alpha}^2}{\abs{\lambda_\alpha-z_0}^{2E+2}}\bm{\omega}_\alpha^{}\bm{\omega}_\alpha^*\text{,}
	\end{equation*}
	and
	\begin{equation*}
	\Ghat=\sum_{\alpha=1}^{N}\frac{\normV{v^\star_\alpha}^2}{\abs{\lambda_\alpha-z_0}^{2E+2}}\bm{\omega}_\alpha^{}\bm{\omega}_\alpha^*\text{.}
	\end{equation*}
	In particular, we remark that $\Gtilde$ is positive definite, due to the linear independence of $\{\bm{\omega}_\alpha\}_{\alpha=1}^\infty$, which, in turn, follows from the fact that the $\{\lambda_\alpha\}_{\alpha=1}^\infty$ are distinct:
	\begin{equation*}
	\qq^*\Gtilde\qq=\sum_{\alpha=1}^\infty\frac{\normV{v^\star_\alpha}^2}{\abs{\lambda_\alpha-z_0}^{2E+2}}\abs{\bm{\omega}_\alpha^*\qq}>0\quad\text{for all }\qq\in\C^{N+1}\setminus\{\bm{0}\}\text{.}
	\end{equation*}
	Due to \eqref{eq:simple_gram} and \eqref{eq:quadratic_simplified_poly}, $\qq_E$ is an eigenvector of $\Gtilde$, corresponding to the minimal eigenvalue, which we denote by $\sigma$, and for which, by employing Lemma~\ref{th:fast_optimality}, we proceed to find an upper bound:
	\begin{equation}\label{eq:den_bound_sigma}
	\sigma=\qq_E^*\Gtilde\qq_E=\widetilde{j}_E\left(\den\right)^2\leq\frac{C'^2}{\abs{\lambda_{N+1}-z_0}^{2E+2}}\text{,}
	\end{equation}
	with $C'$ independent of $E$.
	
	As a preliminary step, we prove a bound for the perturbation $\norm{\Ghat-\Gtilde}{2}$ using the Cauchy-Schwarz inequality:
	\begin{align*}
	\norm{\Ghat-\Gtilde}{2}&=\max_{\qq\in\C^{N+1},\|\qq\|=1}\qq^*\left(\Ghat-\Gtilde\right)\qq\nonumber\\
	&=\max_{\qq\in\C^{N+1},\|\qq\|=1}\sum_{\alpha=N+1}^\infty\frac{\normV{v^\star_\alpha}^2}{\left|\lambda_\alpha-z_0\right|^{2E+2}}\left|\bm{\omega}_\alpha^*\qq\right|^2\\
	&\leq\sum_{\alpha=N+1}^\infty\frac{\normV{v^\star_\alpha}^2}{\left|\lambda_\alpha-z_0\right|^{2E+2}}\norm{\bm{\omega}_\alpha}{2}^2\\
	&=\sum_{\alpha=N+1}^\infty\frac{\normV{v^\star_\alpha}^2}{\left|\lambda_\alpha-z_0\right|^{2E+2}}\sum_{j=0}^{N}\left|\lambda_\alpha-z_0\right|^{2j}\\
	&\leq\sup_{\alpha\geq N+1}\left(\frac{1}{\left|\lambda_\alpha-z_0\right|^2}\sum_{j=0}^{N}\left|\lambda_\alpha-z_0\right|^{2j-2E}\right)\sum_{\alpha=N+1}^\infty\normV{v^\star_\alpha}^2\text{.}
	\end{align*}
	Since $E\geq N$, the supremum is achieved for $\alpha=N+1$. This yields
	\begin{equation}\label{eq:den_bound_norm}
	\norm{\Ghat-\Gtilde}{2}\leq\frac{1}{\left|\lambda_{N+1}-z_0\right|^{2E+2}}\sum_{j=0}^{N}\left|\lambda_{N+1}-z_0\right|^{2j}\sum_{\alpha=N+1}^\infty \normV{v^\star_\alpha}^2=\frac{C''^2}{\left|\lambda_{N+1}-z_0\right|^{2E+2}}\text{,}
	\end{equation}
	with $C''$ independent of $E$. 
	
	Now, let
	\begin{equation*}
	W=\Big[\bm{\omega}_1\big|\cdots\big|\bm{\omega}_N\Big]\in\C^{(N+1)\times N}
	\end{equation*}
	and
	\begin{equation*}
	\Lambda_E=\text{diag}\left(\frac{\normV{v^\star_1}^2}{\abs{\lambda_1-z_0}^{2E+2}},\ldots,\frac{\normV{v^\star_N}^2}{\abs{\lambda_N-z_0}^{2E+2}}\right)\in\C^{N\times N}\text{,}
	\end{equation*}
	so that $\Ghat=W\Lambda_EW^*$. In particular, $W$ is a rank-$N$ matrix, due to the fact that the $\{\lambda_\alpha\}_{\alpha=1}^\infty$ are distinct. As such, it admits a left inverse, i.e. a matrix $W^\dagger\in\C^{N\times(N+1)}$ such that $W^\dagger W=I_{N}$, whose rows we denote by
	\begin{equation*}
	W^\dagger=\Big[\bm{w}_1^\dagger\big|\cdots\big|\bm{w}_{N}^\dagger\Big]^*\text{.}
	\end{equation*}
	
	Now, since $\Gtilde\qq_E=\sigma\qq_E$, \review{we obtain}
	\begin{equation*}
	W\Lambda_EW^*\qq_E=\Ghat\qq_E=\left(\Ghat-\Gtilde\right)\qq_E+\sigma\qq_E\text{.}
	\end{equation*}
	Applying $W^\dagger$ from the left leads to
	\begin{equation*}
	\Lambda_EW^*\qq_E=W^\dagger\left(\Ghat-\Gtilde\right)\qq_E+\sigma W^\dagger\qq_E\text{,}
	\end{equation*}
	i.e., element-wise,
	\begin{equation*}
	\frac{\normV{v^\star_\alpha}^2}{\abs{\lambda_\alpha-z_0}^{2E+2}}\bm{\omega}_\alpha^*\qq_E=\left.\bm{w}_\alpha^\dagger\right.^*\left(\Ghat-\Gtilde\right)\qq_E+\sigma\left.\bm{w}_\alpha^\dagger\right.^*\qq_E\quad\text{for }\alpha=1,\ldots,N\text{.}
	\end{equation*}
	
	Thus, the triangular and Cauchy-Schwarz inequalities, and the normalization of $\qq_E$ lead to
	\begin{align*}
	\abs{\den\left(\lambda_\alpha\right)}=\abs{\bm{\omega}_\alpha^*\qq_E}\leq&\frac{\abs{\lambda_\alpha-z_0}^{2E+2}}{\normV{v^\star_\alpha}^2}\left(\abs{\left.\bm{w}_\alpha^\dagger\right.^*\left(\Ghat-\Gtilde\right)\qq_E}+\sigma\abs{\left.\bm{w}_\alpha^\dagger\right.^*\qq_E}\right)\\
	\leq&\frac{\abs{\lambda_\alpha-z_0}^{2E+2}}{\normV{v^\star_\alpha}^2}\norm{\bm{w}_\alpha^\dagger}{2}\left(\norm{\Ghat-\Gtilde}{2}+\sigma\right)\text{,}
	\end{align*}
	for $\alpha=1,\ldots,N$. The claim follows by exploiting \eqref{eq:den_bound_sigma} and \eqref{eq:den_bound_norm}.
\end{proof}

\review{We are now ready to provide our main result on convergence of fast LS-Pad\'e approximant poles to the $N$ closest poles of $\mcS$.}

\begin{theorem}
    \label{th:pole_convergence}
	Consider the framework of Lemma~\ref{th:den_at_pole_convergence}, and, for fixed $E$, denote the roots of $\den$ by $\{\widetilde{\lambda}_\beta^{(E)}\}_{\beta=1}^{N}$. Then, for $\alpha=1,\ldots,N$, \review{we have that}
    \begin{equation}\label{eq:pole_convergence}
    \min_{\beta=1,\ldots,N}\left|\widetilde{\lambda}_\beta^{(E)}-\lambda_\alpha\right|\leq c'_\alpha\left|\frac{\lambda_\alpha-z_0}{\lambda_{N+1}-z_0}\right|^{2E}\quad\text{for }E\text{ large enough,}
    \end{equation}
    with $c'_\alpha$ independent of $E$.
\end{theorem}

\begin{proof}
	Throughout the proof we assume that $\alpha\in\{1,\ldots,N\}$ is fixed. Also, for fixed $E$, let 
	\begin{equation*}
	\widetilde{\lambda}^{(E)}:\{\lambda_\gamma\}_{\gamma=1}^\infty\to\{\widetilde{\lambda}_\beta^{(E)}\}_{\beta=1}^{N}
	\end{equation*}		
	be the function mapping each pole of $\mcS$ to the closest root of the Pad\'e denominator (in case of ambiguity, any of the closest roots suffices), i.e.
	\begin{equation*}
	\abs{\widetilde{\lambda}^{(E)}(\lambda_\gamma)-\lambda_\gamma}=\min_{\beta=1,\ldots,N}\abs{\widetilde{\lambda}_\beta^{(E)}-\lambda_\gamma}\quad\text{for }\gamma=1,2,\ldots\text{.}
	\end{equation*}
	
	Since $\den$ is normalized, Lemma~\ref{th:interp_bounds} applies, yielding
	\begin{equation}
	\abs{\den(\lambda_\alpha)}\geq\prod_{\beta=1}^{N}\frac{\abs{\widetilde{\lambda}_\beta^{(E)}-\lambda_\alpha}}{1+\abs{\widetilde{\lambda}_\beta^{(E)}-z_0}}\geq\prod_{\beta=1}^{N}\frac{\abs{\widetilde{\lambda}_\beta^{(E)}-\lambda_\alpha}}{1+\abs{\lambda_\alpha-z_0}+\abs{\widetilde{\lambda}_\beta^{(E)}-\lambda_\alpha}}\text{,}\label{eq:den_eval_bound_partial}
	\end{equation}
	thanks to the triangular inequality.
	
	We introduce the strictly increasing continuous function
	\begin{equation}\label{eq:def_phi_alpha}
	\phi_\alpha(x)=\frac{x}{1+\abs{\lambda_\alpha-z_0}+x}\text{,}
	\end{equation}
	defined over the positive real numbers, with $\phi_\alpha(0)=0$ and whose inverse is
	\begin{equation*}	
	\phi_\alpha^{-1}(y)=\left(1+\abs{\lambda_\alpha-z_0}\right)\frac{y}{1-y}
	\end{equation*}
	for $0\leq y<1$.
	
	Now, \eqref{eq:den_eval_bound_partial} and the monotonicity of $\phi_\alpha$ lead to
	\begin{equation*}
	\abs{\den(\lambda_\alpha)}\geq\prod_{\beta=1}^{N}\phi_\alpha\left(\abs{\widetilde{\lambda}_\beta^{(E)}-\lambda_\alpha}\right)\geq\left(\phi_\alpha\left(\abs{\widetilde{\lambda}^{(E)}\left(\lambda_\alpha\right)-\lambda_\alpha}\right)\right)^{N}\text{,}
	\end{equation*}
	so that, \review{thanks to Lemma~\ref{th:den_at_pole_convergence},}
	\begin{equation*}
	\abs{\widetilde{\lambda}^{(E)}\left(\lambda_\alpha\right)-\lambda_\alpha}\leq\phi_\alpha^{-1}\left(\left(c_\alpha\abs{\frac{\lambda_\alpha-z_0}{\lambda_{N+1}-z_0}}^{2E}\right)^{1/N}\right)\text{,}
	\end{equation*}
	provided the argument of $\phi_\alpha^{-1}$ is smaller than 1, i.e. for $E$ large enough. If $\abs{\lambda_\alpha-z_0}=\abs{\lambda_{N+1}-z_0}$, the claim follows trivially by defining $c'_\alpha=\phi_\alpha^{-1}\smash{\big(c_\alpha^{1/N}\big)}$. Thus, for the rest of the proof we assume that $\abs{\lambda_\alpha-z_0}<\abs{\lambda_{N+1}-z_0}$.
	
	Since $c_\alpha$ is independent of $E$, the continuity of $\phi_\alpha^{-1}$ yields
	\begin{equation*}
	\lim_{E\to\infty}\abs{\widetilde{\lambda}^{(E)}\left(\lambda_\alpha\right)-\lambda_\alpha}\leq\phi_\alpha^{-1}\left(\lim_{E\to\infty}\left(c_\alpha\abs{\frac{\lambda_\alpha-z_0}{\lambda_{N+1}-z_0}}^{2E}\right)^{1/N}\right)=0\text{,}
	\end{equation*}
	i.e.
	\begin{equation}\label{eq:den_eval_conv_no_rate}
	\lim_{E\to\infty}\abs{\widetilde{\lambda}^{(E)}\left(\lambda_\alpha\right)-\lambda_\alpha}=0\text{.}
	\end{equation}

	In order to obtain the rate \eqref{eq:pole_convergence}, we define
	\begin{equation*}	
	r=\min_{1\leq\beta<\beta'\leq N}\abs{\lambda_\beta-\lambda_{\beta'}}>0\text{.}
	\end{equation*}
	For $E$ large enough, \eqref{eq:den_eval_conv_no_rate} implies that
	\begin{equation}\label{eq:den_eval_closest_map}
	\abs{\widetilde{\lambda}^{(E)}(\lambda_\gamma)-\lambda_\gamma}<\frac{r}{2}\quad\text{for }\gamma=1,\ldots,N\text{.}
	\end{equation}
	In particular, the approximate poles $\{\widetilde{\lambda}^{(E)}(\lambda_\gamma)\}_{\gamma=1}^{N}$ form a subset of
	\begin{equation*}	
	B=\bigcup_{\gamma=1,\ldots,N}\mathcal{B}\left(\lambda_\gamma,\frac{r}{2}\right)\text{.}
	\end{equation*}
	But $B$ has $N$ disjoint connected components. Thus, thanks to \eqref{eq:den_eval_closest_map}, the map $\smash{\widetilde{\lambda}^{(E)}}$ is injective over $\{\lambda_\gamma\}_{\gamma=1}^{N}$, and \review{we can write}
	\begin{equation*}
	\begin{cases}
	\abs{\widetilde{\lambda}_\beta^{(E)}-\lambda_\alpha}<\frac{r}{2},\quad&\text{if }\widetilde{\lambda}_\beta^{(E)}=\widetilde{\lambda}^{(E)}(\lambda_\alpha),\\[.2cm]
	\abs{\widetilde{\lambda}_\beta^{(E)}-\lambda_\alpha}\geq\frac{r}{2},\quad&\text{for all other }\beta=1,\ldots,N\text{.}
	\end{cases}
	\end{equation*}

	From \eqref{eq:den_eval_bound_partial} it follows that
	\begin{align*}
	\abs{\den(\lambda_\alpha)}&\geq\prod_{\beta=1}^{N}\phi_\alpha\left(\abs{\widetilde{\lambda}_\beta^{(E)}-\lambda_\alpha}\right)\\
	&=\phi_\alpha\left(\abs{\widetilde{\lambda}^{(E)}(\lambda_\alpha)-\lambda_\alpha}\right)\prod_{\substack{\beta=1\\\widetilde{\lambda}_\beta^{(E)}\neq\widetilde{\lambda}^{(E)}(\lambda_\alpha)}}^{N}\phi_\alpha\left(\abs{\widetilde{\lambda}_\beta^{(E)}-\lambda_\alpha}\right)\\
	&\geq\left(\phi_\alpha\left(\frac{r}{2}\right)\right)^{N-1}\phi_\alpha\left(\abs{\widetilde{\lambda}^{(E)}(\lambda_\alpha)-\lambda_\alpha}\right)\text{,}
	\end{align*}
	provided $E$ is large enough.
	
	By Lemma~\ref{th:den_at_pole_convergence} and by applying $\phi_\alpha^{-1}$, it follows that
	\begin{align*}
	\abs{\widetilde{\lambda}^{(E)}(\lambda_\alpha)-\lambda_\alpha}\leq&\,\phi_\alpha^{-1}\left(\left(\phi_\alpha\left(\frac{r}{2}\right)\right)^{1-N}c_\alpha\abs{\frac{\lambda_\alpha-z_0}{\lambda_{N+1}-z_0}}^{2E}\right)\\
	=&\,\frac{\left(1+\abs{\lambda_\alpha-z_0}\right)\left(\phi_\alpha\left(\frac{r}{2}\right)\right)^{1-N}c_\alpha}{1-\left(\phi_\alpha\left(\frac{r}{2}\right)\right)^{1-N}c_\alpha\abs{\frac{\lambda_\alpha-z_0}{\lambda_{N+1}-z_0}}^{2E}}\,\abs{\frac{\lambda_\alpha-z_0}{\lambda_{N+1}-z_0}}^{2E}
	\end{align*}
	for $E$ large enough.
	
	For $E$ large enough, $$\left(\phi_\alpha\left(\frac{r}{2}\right)\right)^{1-N}c_\alpha\abs{\frac{\lambda_\alpha-z_0}{\lambda_{N+1}-z_0}}^{2E}<\frac{1}{2}\text{,}$$ so that
	\begin{equation*}	
	\abs{\widetilde{\lambda}^{(E)}(\lambda_\alpha)-\lambda_\alpha}\leq 2\left(1+\abs{\lambda_\alpha-z_0}\right)\left(\phi_\alpha\left(\frac{r}{2}\right)\right)^{1-N}c_\alpha\abs{\frac{\lambda_\alpha-z_0}{\lambda_{N+1}-z_0}}^{2E}\text{,}
	\end{equation*}
	and the claim \eqref{eq:pole_convergence} follows.
\end{proof}

\begin{corollary}
    \label{th:den_convergence}
	Consider the framework of Lemma~\ref{th:den_at_pole_convergence}, and let $g\in\PspaceN{N}{z_0}$ have roots $\{\lambda_\alpha\}_{\alpha=1}^N$. As $E$ increases, the complex magnitude of the Pad\'e denominator $\abs{\den}$ converges to $\abs{g}$, uniformly over all compact subsets of $\C$.
\end{corollary}

\begin{proof}
	Theorem~\ref{th:pole_convergence} shows that the roots of $\den$, namely $\{\widetilde{\lambda}_\beta^{(E)}\}_{\beta=1}^N$, converge to those of $g$ as $E$ increases. Due to Remark~\ref{re:interp_continuity}, the absolute value of a polynomial in $\PspaceN{N}{z_0}$ depends continuously on its roots, and the claim follows.
\end{proof}

All the results above hold for increasing $E$ with constant denominator degree $N$. A convergence result can be proven also in the case of increasing $N$, as follows.

\begin{theorem}
    \label{th:pole_convergence_N}
	Consider a sequence
	\begin{equation*}	
	\left(E_k,N_k\right)_{k=1}^\infty\subset\{\left(E,N\right)\in\N^2, E\geq N\}\text{,}
	\end{equation*}
	such that $E_{k+1}> E_k$ and $N_{k+1}\geq N_k$ for all $k$.	Let $\denk$ be the fast LS-Pad\'e denominator computed with $E_k$ derivatives of $\mcS$ at $z_0\in\C\setminus\Sigma(\mcL)$, whose roots are denoted by $\{\smash{\widetilde{\lambda}_\beta^{(k)}}\}_{\beta=1}^{N_k}$ (the choice of $M_k$ is irrelevant, as it does not affect $\smash{\widetilde{j}_{E_k}}$). If $\lim_{k\to\infty}N_k=\infty$, then, \review{for all $\alpha=1,2,\ldots$},
    \begin{equation}
    \lim_{k\to\infty}\min_{\beta=1,\ldots,N_k}\abs{\widetilde{\lambda}_\beta^{(k)}-\lambda_\alpha}=0\text{.}
    \end{equation}
\end{theorem}

\begin{proof}
	Let $\alpha\in\{1,2,\ldots\}$ be fixed. Due to \eqref{eq:quadratic_simplified_poly}, \review{we have that}
	\begin{align*}
	\frac{\normV{v^\star_\alpha}}{\abs{\lambda_\alpha-z_0}^{E_k+1}}\abs{\denk(\lambda_\alpha)}\leq&\left(\sum_{\beta=1}^\infty\frac{\normV{v^\star_{\smash\beta}}^2}{\abs{\lambda_\beta-z_0}^{2E_k+2}}\abs{\denk(\lambda_\beta)}^2\right)^{1/2}\\
	=&\widetilde{j}_{E_k}\hspace{-2pt}\left(\denk\right)\text{,}
	\end{align*}
	so that Lemma~\ref{th:fast_optimality} implies
	\begin{equation}\label{eq:den_eval_bound}
	\abs{\denk(\lambda_\alpha)}\leq\frac{\normV{v^\star}}{\normV{v^\star_\alpha}}\prod_{\beta=1}^{N_k}\left(1+\abs{\frac{\lambda_{N_k+1}-z_0}{\lambda_\beta-z_0}}\right)\abs{\frac{\lambda_\alpha-z_0}{\lambda_{N_k+1}-z_0}}^{E_k+1}\text{.}
	\end{equation}
	
	As in the proof of Theorem~\ref{th:pole_convergence}, Lemma~\ref{th:interp_bounds} and the triangular inequality yield
	\begin{align*}
	\abs{\denk(\lambda_\alpha)}\geq&\prod_{\beta=1}^{N_k}\frac{\abs{\widetilde{\lambda}_\beta^{(k)}-\lambda_\alpha}}{1+\abs{\widetilde{\lambda}_\beta^{(k)}-z_0}}\\
	\geq&\prod_{\beta=1}^{N_k}\phi_\alpha\left(\abs{\widetilde{\lambda}_\beta^{(k)}-\lambda_\alpha}\right)\\
	\geq&\phi_\alpha\left(\min_{\beta=1,\ldots,N_k}\abs{\widetilde{\lambda}_\beta^{(k)}-\lambda_\alpha}\right)^{N_k}\text{,}
	\end{align*}
	with $\phi_\alpha$ as in \eqref{eq:def_phi_alpha}. This, together with \eqref{eq:den_eval_bound}, leads to
	\begin{multline}\label{eq:den_N_partial_bound}
	\phi_\alpha\left(\min_{\beta=1,\ldots,N_k}\abs{\widetilde{\lambda}_\beta^{(k)}-\lambda_\alpha}\right)\leq\\\leq\left(\frac{\normV{v^\star}}{\normV{v^\star_\alpha}}\prod_{\beta=1}^{N_k}\left(1+\abs{\frac{\lambda_{N_k+1}-z_0}{\lambda_\beta-z_0}}\right)\abs{\frac{\lambda_\alpha-z_0}{\lambda_{N_k+1}-z_0}}^{E_k+1}\right)^{1/N_k}\text{.}
	\end{multline}
	
	Due to the monotonicity and continuity of $\phi_\alpha$, in order to prove the claim it suffices to show that the right-hand-side of \eqref{eq:den_N_partial_bound} converges to zero as $k$ increases. To this aim, we consider its natural logarithm
	\begin{align*}
	\tau_k^{(\alpha)}&=\frac{1}{N_k}\log\frac{\normV{v^\star}}{\normV{v^\star_\alpha}}+\frac{1}{N_k}\sum_{\beta=1}^{N_k}\log\left(1+\abs{\frac{\lambda_{N_k+1}-z_0}{\lambda_\beta-z_0}}\right)+\frac{E_k+1}{N_k}\log\abs{\frac{\lambda_\alpha-z_0}{\lambda_{N_k+1}-z_0}}\\
	&\leq\frac{1}{N_k}\log\frac{\normV{v^\star}}{\normV{v^\star_\alpha}}+\frac{1}{N_k}\sum_{\beta=1}^{N_k}\log\left(2\abs{\frac{\lambda_{N_k+1}-z_0}{\lambda_\beta-z_0}}\right)+\frac{E_k+1}{N_k}\log\abs{\frac{\lambda_\alpha-z_0}{\lambda_{N_k+1}-z_0}}\\
	&=\frac{1}{N_k}\log\frac{\normV{v^\star}}{\normV{v^\star_\alpha}}+\log 2+\frac{1}{N_k}\sum_{\beta=1}^{N_k}\log\abs{\frac{\lambda_\alpha-z_0}{\lambda_\beta-z_0}}+\frac{E_k+1-N_k}{N_k}\log\abs{\frac{\lambda_\alpha-z_0}{\lambda_{N_k+1}-z_0}}
	\end{align*}
	and prove a bound for each term separately.
	
	\review{Trivially,}
	\begin{equation*}
	\lim_{k\to\infty}\frac{1}{N_k}\log\frac{\normV{v^\star}}{\normV{v^\star_\alpha}}+\log 2=\log 2\text{.}
	\end{equation*}
	Moreover, since $E_k\geq N_k$ for all $k$, the last term satisfies
	\begin{equation*}	
	\frac{E_k+1-N_k}{N_k}\log\abs{\frac{\lambda_\alpha-z_0}{\lambda_{N_k+1}-z_0}}<0
	\end{equation*}
	whenever $\abs{\lambda_\alpha-z_0}<\abs{\lambda_{N_k+1}-z_0}$, i.e. (thanks to the unboundedness of $\{N_k\}_{k=1}^\infty$ and of the spectrum $\Sigma(\mcL)$) for $k$ large enough.
	
	In order to find a bound for the remaining term, we remark that $\smash{\big\{\log\big|\frac{\lambda_\alpha-z_0}{\lambda_\beta-z_0}\big|\big\}_{\beta=1}^\infty}$ is decreasing and unbounded, due, once more, to the unboundedness of the spectrum $\Sigma(\mcL)$. 
	Thus, the Stolz-Ces\`aro theorem \cite{Ash2012} can be applied to a strictly monotone subsequence $(N_{k_l})_{l=1}^\infty$ to prove that
	\begin{equation*}
	\lim_{k\to\infty}\frac{1}{N_k}\sum_{\beta=1}^{N_k}\log\abs{\frac{\lambda_\alpha-z_0}{\lambda_\beta-z_0}}=-\infty\text{.}
	\end{equation*}
	
	In summary, $\lim_{k\to\infty}\tau_k^{(\alpha)}=-\infty$, and the claim follows.
\end{proof}

\begin{remark}
If $\Sigma(\mcL)$ is finite, Lemmas~\ref{th:fast_optimality} and \ref{th:den_at_pole_convergence}, as well as Theorems~\ref{th:pole_convergence} and \ref{th:pole_convergence_N}, and Corollary~\ref{th:den_convergence}, still hold whenever $N<\#\Sigma(\mcL)$, where $\# A$ denotes the cardinality of the set $A$. Also, if $N\geq\#\Sigma(\mcL)$, some of the results become even stronger: within the frameworks of the respective Lemmas and Theorem, \eqref{eq:fast_optimality}-\eqref{eq:den_at_pole_convergence}-\eqref{eq:pole_convergence} become $$\widetilde{j}_E(\den)=0\text{,}$$ and $$\abs{\den(\lambda_\alpha)}=0\quad\text{and}\quad\min_{\beta=1,\ldots,N}\abs{\widetilde{\lambda}_\beta^{(E)}-\lambda_\alpha}=0\quad\text{for }\alpha=1,\ldots,\#\Sigma(\mcL)\text{.}$$
\end{remark}

\begin{remark}
Due to Remark~\ref{re:standard_functional}, all the results in the present Section can be generalized to standard LS-Pad\'e approximants, see Definition~\ref{def:pade_approximant}, whenever the target map $\mcS$ can be expressed using an orthogonal decomposition as in \eqref{eq:meromorphic_decomposition}. However, the main bounds \eqref{eq:fast_optimality}-\eqref{eq:den_at_pole_convergence}-\eqref{eq:pole_convergence} hold only asymptotically in $E$. In particular, numerical tests, see Section~\ref{sec:compare_algorithms}, have shown that, in order to achieve an accuracy which is comparable to that of fast LS-Pad\'e approximants, standard LS-Pad\'e approximants require $N$ more derivatives of the target map $\mcS$.
\end{remark}

\section{Convergence of fast LS-Pad\'{e} approximants}
\label{sec:conv_pade_fast}

Given the results from the previous section, it remains to check whether fast LS-Pad\'{e} approximants inherit the convergence in $V$ from that (in $\C^{N+1}$) of their denominators, and whether their convergence rate is the same as the one for standard LS-Pad\'{e} approximants \eqref{eq:pade_approx_S}.

In this section we prove that fast approximants converge at exponential rate in $M$, provided the denominator degree stays constant. Also, we show that their convergence rate is better than that in \eqref{eq:pade_approx_S}, and is consistent with the numerically observed rate \eqref{eq:pade_approx_S_empirical}.

Moreover, we show that fast LS-Pad\'e approximants converge to the target map $\mcS$ along more general paths of the Pad\'e table, in particular on para-diagonal sequences $[N+\delta/N]$ for $\delta\geq -1$, under some reasonable assumptions on the choice of $E$.

First, we prove a bound for fast LS-Pad\'{e} residuals in terms of both $M$ and $N$.

\begin{lemma}
    \label{th:fast_convergence_H}
    For any $E,M,N\in\N$, with $M\geq N-1$ and $E=\max\{M,N\}$, consider the (meromorphic) fast LS-Pad\'e residual $\HMN:\C\setminus\Sigma(\mcL)\to V$, defined as
    \begin{equation}
    \label{eq:H_definition}
    \HMN=\den\mcS-\num\text{.}
    \end{equation}
    
	For $z\in\C$, let
	\begin{equation*}
    \distS(z)=\min_{\lambda\in\Sigma(\mcL)}\left|\lambda-z\right|\text{.}
	\end{equation*}
	Then, for all $z\in\C\setminus\Sigma(\mcL)$, \review{we have the bounds:}
    \begin{equation}
    \label{eq:H_bound}
    \normV{\HMN(z)}\leq\frac{C'}{\distS(z)}\abs{\frac{z-z_0}{\lambda_{N+1}-z_0}}^{E+1}\text{if }M\geq N\text{,}
    \end{equation}
	and
    \begin{equation}
    \label{eq:H_bound_bis}
    \normV{\HMN(z)}\leq C'\left(\frac{1}{\distS(z)}+\frac{1}{\left|z-z_0\right|}\right)\abs{\frac{z-z_0}{\lambda_{N+1}-z_0}}^{E+1}\text{if }M=N-1\text{.}
    \end{equation}
    In particular, the common constant $C'$ is given by \eqref{eq:fast_optimality_const}.
\end{lemma}

\begin{proof}
	We can exploit \eqref{eq:meromorphic_decomposition_taylor} to derive
	\begin{equation}
    \label{eq:QS_definition}
	\den(z)\mcS(z)=\sum_{\alpha=1}^\infty\frac{v^\star_\alpha}{\lambda_\alpha-z}\den(z)\text{.}
	\end{equation}

	Due to Definition~\ref{def:pade_simplified} and \eqref{eq:meromorphic_decomposition_taylor}, \review{we can express the fast Pad\'e numerator as}
    \begin{align*}
	\num(z)=&\sum_{j=0}^M\coeff{\den\mcS}{j,z_0}\left(z-z_0\right)^j=\sum_{j=0}^M\sum_{l=0}^j\coeff{\den}{l,z_0}\coeff{\mcS}{j-l,z_0}\left(z-z_0\right)^j\\
	=&\sum_{j=0}^M\sum_{l=0}^j\coeff{\den}{l,z_0}\sum_{\alpha=1}^\infty\frac{v^\star_\alpha}{\left(\lambda_\alpha-z_0\right)^{j-l+1}}\left(z-z_0\right)^j\\
	=&\sum_{\alpha=1}^\infty\frac{v^\star_\alpha}{\lambda_\alpha-z_0}\sum_{l=0}^M\coeff{\den}{l,z_0}\left(z-z_0\right)^l\sum_{j=l}^M\left(\frac{z-z_0}{\lambda_\alpha-z_0}\right)^{j-l}\\
	=&\sum_{\alpha=1}^\infty\frac{v^\star_\alpha}{\lambda_\alpha-z_0}\sum_{l=0}^M\coeff{\den}{l,z_0}\left(z-z_0\right)^l\frac{\left(\frac{z-z_0}{\lambda_\alpha-z_0}\right)^{M-l+1}-1}{\frac{z-z_0}{\lambda_\alpha-z_0}-1}\\
	=&\sum_{\alpha=1}^\infty\frac{v^\star_\alpha}{\lambda_\alpha-z}\sum_{l=0}^M\coeff{\den}{l,z_0}\left(z-z_0\right)^l\left(1-\left(\frac{z-z_0}{\lambda_\alpha-z_0}\right)^{M-l+1}\right)\text{.}
	\end{align*}	
	\review{Under our hypotheses, we can replace the upper summation index $M$ in the last sum by $N$. Indeed, this is trivially true for $M\geq N$, since $\coeff{\den}{l,z_0}=0$ for $l>N$. In the case $M=N-1$, direct inspection shows that the addend corresponding to $l=N$ is zero, thus justifying its addition to the sum. Hence, the fast Pad\'e numerator can be expressed as}
	\begin{align*}
	\num(z)=&\sum_{\alpha=1}^\infty\frac{v^\star_\alpha}{\lambda_\alpha-z}\sum_{l=0}^N\coeff{\den}{l,z_0}\left(z-z_0\right)^l\left(1-\left(\frac{z-z_0}{\lambda_\alpha-z_0}\right)^{M-l+1}\right)\\
	=&\sum_{\alpha=1}^\infty\frac{v^\star_\alpha}{\lambda_\alpha-z}\left(\den(z)-\sum_{l=0}^N\coeff{\den}{l,z_0}\left(\lambda_\alpha-z_0\right)^l\left(\frac{z-z_0}{\lambda_\alpha-z_0}\right)^{M+1}\right)\\
	=&\sum_{\alpha=1}^\infty\frac{v^\star_\alpha}{\lambda_\alpha-z}\left(\den(z)-\left(\frac{z-z_0}{\lambda_\alpha-z_0}\right)^{M+1}\den(\lambda_\alpha)\right)\\
	=&\den(z)\mcS(z)-\sum_{\alpha=1}^\infty\frac{v^\star_\alpha}{\lambda_\alpha-z}\den(\lambda_\alpha)\left(\frac{z-z_0}{\lambda_\alpha-z_0}\right)^{M+1}\text{,}
	\end{align*}
	\review{see \eqref{eq:QS_definition}.}



	\review{Thus, by exploiting \eqref{eq:H_definition} and the $V$-orthogonality of $\{v^\star_\alpha\}_{\alpha=1}^\infty$, we can express the squared norm of the residual as}
	\begin{equation}
	\label{eq:H_residual_sum}
	\normV{\HMN(z)}^2=\left|z-z_0\right|^{2M+2}\sum_{\alpha=1}^\infty\frac{\normV{v^\star_\alpha}^2}{\left|\lambda_\alpha-z\right|^2\left|\lambda_\alpha-z_0\right|^{2M+2}}\left|\den\left(\lambda_\alpha\right)\right|^2\text{.}
	\end{equation}

	We distinguish two cases:
	\begin{itemize}
	\item\textbf{Case $E=M\geq N$.}
	From \eqref{eq:H_residual_sum}, by exploiting \eqref{eq:quadratic_simplified_poly} we can derive
	\begin{align*}
	\normV{\HMN(z)}^2\leq&\left|z-z_0\right|^{2M+2}\frac{1}{\inf_{\lambda\in\Sigma(\mcL)}\left|\lambda-z\right|^2}\sum_{\alpha=1}^\infty\frac{\normV{v^\star_\alpha}^2}{\left|\lambda_\alpha-z_0\right|^{2M+2}}\left|\den\left(\lambda_\alpha\right)\right|^2\\
	=&\left|z-z_0\right|^{2M+2}\frac{1}{\distS(z)^2}\sum_{\alpha=1}^\infty\frac{\normV{v^\star_\alpha}^2}{\left|\lambda_\alpha-z_0\right|^{2M+2}}\left|\den\left(\lambda_\alpha\right)\right|^2\\
	=&\left|z-z_0\right|^{2M+2}\frac{1}{\distS(z)^2}\widetilde{j}_E\left(\den\right)^2\text{.}
	\end{align*}
	
	Lemma~\ref{th:fast_optimality} can now be applied, leading to
	\begin{equation*}
	\normV{\HMN(z)}^2\leq\frac{\left.C'\right.^2}{\distS(z)^2}\abs{\frac{z-z_0}{\lambda_{N+1}-z_0}}^{2E+2}\text{.}
	\end{equation*} 
	
	\item\textbf{Case $E=N=M+1$.}
	Equation \eqref{eq:H_residual_sum} can be written equivalently as
	\begin{equation*}
	\normV{\HMN(z)}^2=\left|z-z_0\right|^{2M+2}\sum_{\alpha=1}^\infty\frac{\normV{v^\star_\alpha}^2}{\left|\lambda_\alpha-z_0\right|^{2M+4}}\left|\den\left(\lambda_\alpha\right)\right|^2\left|\frac{\lambda_\alpha-z_0}{\lambda_\alpha-z}\right|^2\text{.}
	\end{equation*}

	Now we observe that, \review{for any $\alpha\geq 1$,}
	\begin{equation*}
	\left|\frac{\lambda_\alpha-z_0}{\lambda_\alpha-z}\right|\leq\frac{\left|\lambda_\alpha-z\right|+\left|z-z_0\right|}{\left|\lambda_\alpha-z\right|}=1+\left|\frac{z-z_0}{\lambda_\alpha-z}\right|\leq 1+\frac{\left|z-z_0\right|}{\distS(z)}\text{,}
	\end{equation*} 
	which yields
	\begin{equation*}
	\normV{\HMN(z)}^2\leq\left|z-z_0\right|^{2M+2}\left(1+\frac{\left|z-z_0\right|}{\distS(z)}\right)^2\widetilde{j}_E\left(\den\right)^2\text{.}
	\end{equation*} 
	
	To conclude, it suffices to apply Lemma~\ref{th:fast_optimality}:
	\begin{equation*}
	\normV{\HMN(z)}^2\leq\left.C'\right.^2\left(\frac{1}{\distS(z)}+\frac{1}{\left|z-z_0\right|}\right)^2\abs{\frac{z-z_0}{\lambda_{N+1}-z_0}}^{2E+2}\text{.}
	\end{equation*}
	\end{itemize}
\end{proof}

\begin{remark}
If $\Sigma(\mcL)$ is finite, Lemma~\ref{th:fast_convergence_H} still holds true whenever $N<\#\Sigma(\mcL)$. Moreover, $\normV{\HMN}=0$ over all $\C\setminus\Sigma(\mcL)$ if $M+1\geq N\geq\#\Sigma(\mcL)$.
\end{remark}

Finally, we can use the previous results to prove the convergence in measure of fast LS-Pad\'{e} approximants within the region of the Pad\'e table where $M\geq N-1$ and $E=\max\{M,N\}$.

\begin{theorem}
    \label{th:fast_convergence}
    Let $z_0\in\C\setminus\Sigma(\mcL)$ and $R>0$ be fixed, so that no pole of $\mcS$ lies on $\partial\ball{z_0}{R}$. Also, let $\overline N\in\N$ be the number of poles of $\mcS$ within $\ball{z_0}{R}$. Consider a sequence
	\begin{equation*}    
    \left(M_k,N_k\right)_{k=1}^\infty\subset\{\left(M,N\right)\in\N^2, M\geq N-1\}\text{,}
	\end{equation*}
	such that $M_{k+1}>M_k$ and $N_{k+1}\geq N_k$ for all $k$, with $\lim_{k\to\infty}N_k\geq\overline N$.
    
    Let $\padek$ be the $[M_k/N_k]$ fast LS-Pad\'e approximant of $\mcS$, computed with $E=\max\{M_k,N_k\}$ for $k=1,2,\ldots$. \review{For any $\varepsilon>0$,}
    \begin{equation}
    \label{eq:fast_convergence}
    \lim_{k\to\infty}\abs{\left\{z\in\ball{z_0}{R}\;:\;\normV{\mcS(z)-\padek(z)}>\varepsilon\right\}}=0\text{,}
    \end{equation}
    with $\abs{A}$ denoting the Lebesgue measure of the set $A$.
\end{theorem}

\begin{proof}
	Let $k$ be fixed. We indicate with $\smash{\{\widetilde\lambda_\alpha\}_{\alpha=1}^{N_k}}$ the roots of $\denk$, ordered with respect to their distance from $z_0$, and we consider the integer $N'_k\in\{0,\ldots,N_k\}$ such that
	\begin{equation}
	\label{eq:fast_den_roots1}
	|\widetilde\lambda_\alpha-z_0|\leq 2R\quad\text{for }\alpha=1,\ldots,N'_k\text{,}
	\end{equation}
	and
	\begin{equation}
	\label{eq:fast_den_roots2}
	|\widetilde\lambda_\alpha-z_0|>2R\quad\text{for }\alpha=N'_k+1,\ldots,N_k\text{.}
	\end{equation}

	Since $\denk$ belongs to $\PspaceN{N_k}{z_0}$, Lemma~\ref{th:interp_bounds} applies, yielding
	\begin{equation*}
	\abs{\denk(z)}\geq\prod_{\alpha=1}^{N_k}\frac{|\widetilde\lambda_\alpha-z|}{1+|\widetilde\lambda_\alpha-z_0|}\text{.}
	\end{equation*}
	In order to prove a lower bound for $\left|\denk\right|$ over a suitable subset of $\ball{z_0}{R}$, we consider each factor separately. For the terms corresponding to $1\leq\alpha\leq N'_k$, by \eqref{eq:fast_den_roots1} \review{we can write}
	\begin{equation*}
	\frac{|\widetilde\lambda_\alpha-z|}{1+\abso{\widetilde\lambda_\alpha-z_0}}\geq\frac{\abso{\widetilde\lambda_\alpha-z}}{1+2R}\text{.}
	\end{equation*}
	To find a bound for the factors for $N'_k+1\leq\alpha\leq N_k$, we remark that the function $\psi(x)=x/(1+x)$ is increasing for $x>0$. This, together with the triangular inequality and \eqref{eq:fast_den_roots2}, for all $z\in\ball{z_0}{R}$ leads to
	\begin{align*}
	\frac{|\widetilde\lambda_\alpha-z|}{1+\abso{\widetilde\lambda_\alpha-z_0}}\geq&\frac{\abso{\widetilde\lambda_\alpha-z_0}}{1+\abso{\widetilde\lambda_\alpha-z_0}}-\frac{\abso{z-z_0}}{1+\abso{\widetilde\lambda_\alpha-z_0}}\\
	\geq&\frac{2R}{1+2R}-\frac{R}{1+2R}=\frac{R}{1+2R}\text{.}
	\end{align*}
	
	In summary, \review{we have the bound}
	\begin{equation*}
	\abs{\denk(z)}\geq\frac{R^{N_k-N'_k}}{\left(1+2R\right)^{N_k}}\prod_{\alpha=1}^{N'_k}\abso{\widetilde\lambda_\alpha-z}=\frac{R^{N_k-N'_k}}{\left(1+2R\right)^{N_k}}\abs{\ell_{N'_k}(z)}
	\end{equation*}
	for all $z\in\ball{z_0}{R}$, with $\ell_{N'_k}$ being a monic polynomial of degree $N'_k$.
	
	For any fixed $0<\delta_k'\leq\pi R^2$, classical results on lemniscates for monic polynomials (see e.g.~\cite[Theorems 6.6.3--6.6.4]{Baker1996}) prove the existence of a set $\mathcal{E}_k'\subset\C$, with Lebesgue measure $|\mathcal{E}_k'|\leq\delta_k'$, such that 
	\begin{equation*}	
	\abs{\ell_{N'_k}(z)}\geq\left(\frac{\delta_k'}{\pi}\right)^{N'_k/2}\text{ for all }z\in\C\setminus\mathcal{E}_k'\text{.}
	\end{equation*}
	Hence, \review{for all $z\in\ball{z_0}{R}\setminus\mathcal{E}_k'$,}
	\begin{equation}
	\label{eq:fast_den_measure_bound}
	\abs{\denk(z)}\geq\left(\frac{R}{1+2R}\right)^{N_k}\left(\frac{\sqrt{\delta_k'/\pi}}{R}\right)^{N'_k}\geq\left(\frac{\sqrt{\delta_k'/\pi}}{1+2R}\right)^{N_k}\text{.}
	\end{equation}
	
    Now, let $z\in\ball{z_0}{R}\setminus\left(\mathcal{E}_k'\cup\Sigma(\mcL)\right)$ and assume $M_k\geq N_k$; the case $M_k=N_k-1$ can be treated in an analogous way. Lemma~\ref{th:fast_convergence_H}, together with \eqref{eq:fast_den_measure_bound}, yields
    \begin{align*}
    \normV{\mcS(z)-\padek(z)}=&\frac{1}{\left|\denk(z)\right|}\normV{\HMNk(z)}\\
    \leq&\frac{C'}{\distS(z)\left|\denk(z)\right|}\abs{\frac{z-z_0}{\lambda_{N_k+1}-z_0}}^{M_k+1}\\
    \leq&\frac{C'}{\distS(z)}\left(\frac{1+2R}{\sqrt{\delta_k'/\pi}}\right)^{N_k}\abs{\frac{z-z_0}{\lambda_{N_k+1}-z_0}}^{M_k+1}\text{,}
    \end{align*}
    with $C'$ as in Lemma~\ref{th:fast_convergence_H}.
    
    The term $1/\distS(z)$ diverges as $z$ gets close to $\Sigma(\mcL)$. As such, we proceed by excluding small neighborhoods of the poles of $\mcS$ within the region of convergence. To this aim, let $0<\delta_k''<\overline{N}\pi\smash{\big(\abs{\lambda_{\overline N+1}-z_0}-R\big)^2}$ be given. The set
	\begin{equation*}
    \mathcal{E}_k''=\bigcup_{\alpha=1,\ldots,\overline N}\mathcal{B}\left(\lambda_\alpha, \sqrt{\frac{\delta_k''}{\overline{N}\pi}}\right)
	\end{equation*}
    has Lebesgue measure $\abs{\mathcal{E}_k''}\leq\delta_k''$ and satisfies
	\begin{equation*}
	\distS(z)\geq\sqrt{\frac{\delta_k''}{\overline{N}\pi}}\quad\text{for all }z\in\ball{z_0}{R}\setminus\mathcal{E}_k''\text{.}
    \end{equation*}
    In particular, we remark that, thanks to the ordering of the elements of $\Sigma(\mcL)$, the condition $\delta_k''<\overline{N}\pi\smash{\big(\abs{\lambda_{\overline N+1}-z_0}-R\big)^2}$ allows to ignore all the poles with distance from $z_0$ larger than $R$ in the estimation of $\distS$ over $\ball{z_0}{R}$.
    
    If we define $\mathcal{E}_k=\mathcal{E}_k'\cup\mathcal{E}_k''$, whose measure is not greater than $\delta_k'+\delta_k''$ by construction, for all $z\in\ball{z_0}{R}\setminus\mathcal{E}_k$, \review{we have that}
    \begin{multline}
    \label{eq:fast_convergence_ineq_0}
    \normV{\mcS(z)-\padek(z)}\leq\frac{\sqrt{\hspace{.01cm}\overline{N}\pi}\normV{v^\star}}{\sqrt{\left.\delta_k'\right.^{N_k}\delta_k''}}\abs{\frac{z-z_0}{\lambda_{N_k+1}-z_0}}^{M_k+1}\times\\
    \times\left(\frac{|\lambda_{N_k+1}-z_0|}{R}\right)^{N_k}\prod_{\alpha=1}^{N_k}\left(\left(\sqrt{\pi}R(1+2R)\right)\left(\frac{1}{\abs{\lambda_{N_k+1}-z_0}}+\frac{1}{\abs{\lambda_\alpha-z_0}}\right)\right)\text{,}
    \end{multline}
    which, by exploiting the ordering of the poles $\{\lambda_\alpha\}_{\alpha=1}^\infty$, implies
    \begin{multline}
    \label{eq:fast_convergence_ineq}
    \normV{\mcS(z)-\padek(z)}\leq\frac{\sqrt{\hspace{.01cm}\overline{N}\pi}\normV{v^\star}}{\sqrt{\left.\delta_k'\right.^{N_k}\delta_k''}}\left(\frac{R}{|\lambda_{N_k+1}-z_0|}\right)^{M_k+1-N_k}\times\\
    \times\prod_{\alpha=1}^{N_k}\frac{2\sqrt{\pi}R(1+2R)}{\abs{\lambda_\alpha-z_0}}\text{.}
    \end{multline}
    
	To conclude the proof we consider two cases:    
	\begin{itemize}
	\item\textbf{Case $(N_k)_{k=1}^\infty$ bounded.}
	There exists $K>0$ such that $N_k=\lim_{l\to\infty}N_l=:N^\star$ for $k\geq K$. For $k\geq K$, \eqref{eq:fast_convergence_ineq} can be expressed as
    \begin{equation}
    \label{eq:fast_convergence_ineq1}
    \normV{\mcS(z)-\mcS_{[M_k/N^\star]}(z)}\leq\frac{C}{\sqrt{\left.\delta_k'\right.^{N^\star}\delta_k''}}\left(\frac{R}{|\lambda_{N^\star+1}-z_0|}\right)^{M_k+1}
    \end{equation}
	for all $z\in\ball{z_0}{R}\setminus\mathcal{E}_k$, with $C$ independent of $k$. Since $R<\abs{\lambda_{N^\star+1}-z_0}$ and $\lim_{k\to\infty}M_k=\infty$, \review{we can easily see that}
	\begin{equation*}
	\lim_{k\to\infty}\left(\frac{R}{|\lambda_{N^\star+1}-z_0|}\right)^{M_k+1}=0\quad\text{for all }z\in\ball{z_0}{R}\text{.}
	\end{equation*}
	
	For all $k\geq K$, let
	\begin{equation*}
	\delta_k'=\min\left\{\pi R^2,\left(\frac{C}{\varepsilon}\left(\frac{R}{|\lambda_{N^\star+1}-z_0|}\right)^{M_k+1}\right)^{1/N^\star}\right\}
	\end{equation*}
	and
	\begin{equation*}
	\delta_k''=\min\left\{\overline{N}\pi\big(\abs{\lambda_{\overline N+1}-z_0}-R\big)^2,\frac{C}{\varepsilon}\left(\frac{R}{|\lambda_{N^\star+1}-z_0|}\right)^{M_k+1}\right\}\text{.}
	\end{equation*}
	With these definitions, \eqref{eq:fast_convergence_ineq1} implies that
	\begin{equation*}	
	\normV{\mcS(z)-\mcS_{[M_k/N^\star]}(z)}\leq\varepsilon\quad\text{for all }z\in\ball{z_0}{R}\setminus\mathcal{E}_k\text{,}
	\end{equation*}	
	with $\abs{\mathcal{E}_k}\leq\delta_k'+\delta_k''$. As both $\delta_k'$ and $\delta_k''$ converge to 0 as $k$ increases, the claim follows.

	\item\textbf{Case $(N_k)_{k=1}^\infty$ unbounded.}
	As in the previous case, we leverage \eqref{eq:fast_convergence_ineq} to obtain suitable definitions for $\delta_k'$ and $\delta_k''$: for all $k=1,2,\ldots$, we set
	\begin{equation*}
	\delta_k'=\min\left\{\pi R^2,\left(\prod_{\alpha=1}^{N_k}\frac{2\sqrt{\pi}R(1+2R)}{|\lambda_\alpha-z_0|}\right)^{2/N_k}\right\}
	\end{equation*}
	and	
	\begin{equation*}
	\delta_k''=\min\left\{\overline{N}\pi\big(\abs{\lambda_{\overline N+1}-z_0}-R\big)^2,\frac{C^2}{\varepsilon^2}\left(\frac{R}{\abs{\lambda_{N_k+1}-z_0}}\right)^{2(M_k+1-N_k)}\right\}\text{.}
	\end{equation*}
	As before, $\normV{\mcS(z)-\padek(z)}\leq\varepsilon$ for all $z\in\ball{z_0}{R}\setminus\mathcal{E}_k$, with $\abs{\mathcal{E}_k}\leq\delta_k'+\delta_k''$. To prove the claim, it now suffices to show that $\delta_k'$ and $\delta_k''$ converge to 0 as $k$ increases.
	
	\review{Let us consider $\delta_k'$ first}: for all $k$ \review{we have that}
	\begin{equation*}
	\delta_k'\leq\text{exp}\left\{\frac{2}{N_k}\sum_{\alpha=1}^{N_k}\log\frac{2\sqrt{\pi}R(1+2R)}{|\lambda_\alpha-z_0|}\right\}\text{.}
	\end{equation*}
	Since the spectrum $\Sigma(\mcL)$ has a single limit point at infinity, \review{we obtain}
	\begin{equation*}	
	\lim_{\alpha\to\infty}\log\frac{2\sqrt{\pi}R(1+2R)}{|\lambda_\alpha-z_0|}=-\infty\text{.}
	\end{equation*}	
	Now, since $(N_k)_{k=1}^\infty$ is non-decreasing and unbounded, the Stolz-Ces\`aro theorem \cite{Ash2012} can be applied to a strictly monotone subsequence $(N_{k_l})_{l=1}^\infty$ to prove that
	\begin{equation*}	
	\lim_{k\to\infty}\frac{2}{N_k}\sum_{\alpha=1}^{N_k}\log\frac{2\sqrt{\pi}R(1+2R)}{|\lambda_\alpha-z_0|}=-\infty\text{,}
	\end{equation*}	
	or, equivalently, that $\lim_{k\to\infty}\delta_k'=0$.
	
	The second parameter $\delta_k''$ is easier to deal with: since $M_k\geq N_k$ for all $k$, the convergence of $\delta_k''$ to 0 can be verified by exploiting once more the unboundedness of the spectrum $\Sigma(\mcL)$.
	\end{itemize}
\end{proof}

\begin{corollary}
    \label{th:fast_convergence_speed}
	Assume that the hypotheses of Theorem~\ref{th:fast_convergence} hold with $N_k=N^\star\geq\overline N$ for all $k$.
	For any $\delta>0$ there exist $C''$ independent of $k$ and of $z$, and $\mathcal{E}_k\subset\C$, with $\left|\mathcal{E}_k\right|\leq\delta$, such that, for all $z\in\ball{z_0}{R}\setminus\mathcal{E}_k$, \review{the approximation error admits the following bound:}
    \begin{equation}
    \label{eq:fast_convergence_speed}
    \normV{\mcS(z)-\padegen{\mcS}{M_k}{N^\star}(z)}\leq C''\abs{\frac{z-z_0}{\lambda_{N^\star+1}-z_0}}^{M_k}\text{.}
    \end{equation}
\end{corollary}

\begin{proof}
	The claim follows from \eqref{eq:fast_convergence_ineq_0}.
\end{proof}

\begin{remark}
Theorem~\ref{th:fast_convergence} and Corollary~\ref{th:fast_convergence_speed} still hold if $\Sigma(\mcL)$ is finite. In particular, if $\lim_{k\to\infty}N_k\geq\#\Sigma(\mcL)$, both results are satisfied by setting 
$\mathcal{E}_k=\Sigma(\mcL)\cap\ball{z_0}{R}$,
and the right hand side of \eqref{eq:fast_convergence_speed} is identically 0 for large $k$.
\end{remark}

\begin{remark}
The sequence of sets $\{\mathcal{E}_k\}_{k=1}^\infty$ in the proof of Theorem~\ref{th:fast_convergence} and in Corollary~\ref{th:fast_convergence_speed} is used to account for the instabilities of the solution map $\mcS$ and of the rational approximant $\pade$ near the respective poles. In particular, the proof of Theorem~\ref{th:fast_convergence} shows that each $\mathcal{E}_k$ can be defined as the union of suitable neighborhoods of poles of $\mcS$ and of $\pade$.
\end{remark}

\begin{remark}
\review{With a small effort (the necessary theoretical tools can be found, e.g., in \cite[Section 6.6]{Baker1996}), Theorem~\ref{th:fast_convergence} can be extended to show that \eqref{eq:fast_convergence} still holds true if logarithmic capacity \cite{Baker1996, Hille1965} replaces Lebesgue measure. Similarly, the sets in the family $\{\mathcal{E}_k\}_{k=1}^\infty$ in Corollary~\ref{th:fast_convergence_speed} can be shown to have arbitrarily small logarithmic capacity. In this way, optimal convergence results in classical Pad\'e approximation \cite{Baker1996} find their counterparts for fast LS-Pad\'e approximants.}
\end{remark}

\section{Numerical implementation of fast LS-Pad\'{e} approximants}
\label{sec:algo_details}

In this section, we give some details on the practical implementation of Algorithm~\ref{algo:fast}. Consider a compact set $K\subset\C$ where we wish to approximate the meromorphic map $\mcS$. To guarantee the convergence of LS-Pad\'e approximants in $K\setminus\Sigma(\mcL)$, we must choose $z_0\in\C\setminus\Sigma(\mcL)$ and estimate the number $\overline{N}\in\N$ of poles contained in the smallest disk which includes $K$. Still, in most applications, $\Sigma(\mcL)$ is not known explicitly. Hence, a preliminary approximate localization of $\Sigma(\mcL)$ (or, at least, of the elements of $\Sigma(\mcL)$ closest to $K$) is necessary.

A description or analysis of such a procedure falls outside the scope of this paper. However, we envision two possible strategies:
\begin{itemize}
    \item the number of elements of $\Sigma(\mcL)$ within a certain real interval can be approximated through \emph{a priori} eigenvalue estimators, e.g. by applying Weyl's law, see~\cite{Bhatia2007};
    \item an estimate of the positions of the poles of $\mcS$ closest to $z_0$ may be obtained adaptively through the application of fast LS-Pad\'e approximants, where the value of $N$ is updated according to some \emph{a posteriori} estimator computed from Pad\'e denominators.
\end{itemize}

From now on, we assume that $z_0$ and the denominator degree $N$ have been fixed. For instance, we may have set $z_0$ equal to the Chebyshev center of $K$, i.e. the center of the smallest ball which contains $K$. Moreover, we assume that $N$ is not smaller than $\overline{N}$, so that $K\subset\ball{z_0}{\left|\lambda_{N+1}-z_0\right|}$, where, as usual, we order the elements of $\Sigma(\mcL)$ with respect to their distance to $z_0$.

Finally, it is necessary to choose $M$ and $E$, with $M\geq N-1$ and $E=\max\{M,N\}$; this last condition is to ensure that Theorem~\ref{th:fast_convergence} and Corollary~\ref{th:fast_convergence_speed} can be applied. The value of $E$ represents the number of derivatives of $\mcS$ that need to be computed, and affects the accuracy of the approximation of the poles of $\mcS$, see Section~\ref{sec:conv_denominator_fast}. However, while a larger $E$ is expected to yield a better approximation of the exact denominator $g$, in practice it may be desirable to choose a smaller value, since the condition number of $\Gtilde$ increases exponentially with $E$, leading to numerical instability (see also~\cite{Guillaume1998} for similar observations in the case of least-squares multivariate scalar Pad\'e approximants).

Once the Taylor coefficients of $\mcS$ at $z_0$, i.e. $\{\coeff{\mcS}{\gamma,z_0}\}_{\gamma=0}^E$, are computed by exploiting \eqref{eq:problem_parametric} and \eqref{eq:helmholtz_taylor1_parametric_weak}, the functional $\smash{\widetilde{j}_{E}}$ needs to be minimized. To this aim, instead of building explicitely the matrix $\Gtilde$, its Gram structure is exploited to obtain a better conditioned problem. In particular, the quasi-matrix
\begin{equation*}
{\bf\mathcal{A}}=\left[\coeff{\mcS}{E-N,z_0}\Big|\cdots\Big|\coeff{\mcS}{E,z_0}\right]\text{,}
\end{equation*}
whose range is a subspace of $V$, is assembled, and its QR decomposition is computed \cite{Trefethen2009}, so that 
\begin{equation}
\label{eq:quasimatrix_qr}
{\bf\mathcal{A}}=\left[\mcQ_{E-N}\Big|\cdots\Big|\mcQ_{E}\right]R\text{,}
\end{equation}
with $\{\mcQ_j\}_{j=E-N}^E\subset V$ forming a $V$-orthonormal set, and $R\in\C^{(N+1)\times(N+1)}$ being upper triangular. This allows us to find the denominator $\den$ from a right-singular vector of $R$ corresponding to the minimal singular value, effectively with a condition number which is the square root of the one for the original problem.

In many applications (for instance -- and in particular -- in the field of model order reduction for parametric PDEs), both $V$ and $\mcL$ are actually finite-dimensional approximations of some reference infinite-dimensional space $V^0$ and operator $\mcL^0$ respectively, see Section~\ref{sec:compare_algorithms} for an example. This does not affect the results discussed in the previous sections, but introduces an additional source of error, namely the approximation of the PDE, which is not considered in this work.

In this particular but quite common framework, the evaluation of the target map through the solution of \eqref{eq:problem_parametric} and the recursion \eqref{eq:helmholtz_taylor1_parametric_weak} correspond to the solution of linear systems, whose matrices depend parametrically on $z$. Thus, the derivatives of $\mcS$ can be interpreted as a basis of the Krylov subspace of $V$ generated by $\smash{\big((\mcL-z_0\mcI)^{-1},v^\star\big)}$. As such, an approach based on the Arnoldi algorithm could be applied to obtain quite naturally the orthogonal decomposition \eqref{eq:quasimatrix_qr}.

\section{Numerical comparison of standard and fast LS-Pad\'{e} approximants}
\label{sec:compare_algorithms}

We devote this section to the comparison of standard and fast LS-Pad\'{e} approximants for the map $\mcS$ which associates to any value of $z$ the $\mathbb{P}^3$ finite element discretization of the self-adjoint Helmholtz problem \eqref{eq:helmholtz_parametric_weak_complex}, with $\Omega=(0,\pi)^2$ and $f\in L^2(\Omega)$. We refer to \cite{Bonizzoni2018} for further numerical examples of (standard) LS-Pad\'e approximation in similar and more general (non-self-adjoint) settings.

In particular, given $\nu\in\R^+$ and $\theta\in[0,2\pi)$, we define $\bm{d}=(\cos(\theta),\sin(\theta))^\top$ and
\begin{equation*}
u_{ex}(\x)=w(\x)e^{-i\nu\bm{d}^\top\x}\in H^1_0(\Omega)\text{,}
\end{equation*}
with $w(\x)=\frac{16}{\pi^4}x_1x_2(\pi-x_1)(\pi-x_2)$ being a bubble vanishing on $\partial\Omega$. Moreover, we set $f=-\Delta u_{ex}-\nu^2u_{ex}$, so that $u_{ex}=\mcS(\nu^2)$. For our numerical experiments, we choose $\nu^2=12$ and $\theta=\frac{\pi}{3}$.

As described in Section~\ref{sec:helmholtz_solmap}, the spectral decomposition \eqref{eq:meromorphic_decomposition} holds true, with $v^\star=f$. In particular, as our experiments will be carried out in a finite element framework, it is crucial to remark that a finite (and finite-dimensional) counterpart of \eqref{eq:meromorphic_decomposition} is true in the discrete setting as well. Moreover, the discrete spectrum of the Laplacian is a good approximation of the infinite-dimensional one, at least for low/mid-frequencies (here the adjectives ``low/mid'' have to be understood in a relative sense with respect to the specific meshsize and finite element degree which are employed \cite{Babuska1997}).

Hence, as the solution map $\mcS$ is meromorphic (both in the continuous and discrete settings), we wish to approximate it for $z$ within the interval of interest $K=[9,15]$ using LS-Pad\'e approximants, according to Definitions~\ref{def:pade_approximant} and \ref{def:pade_simplified}. As discussed in Section~\ref{sec:helmholtz_solmap}, the problem of computing LS-Pad\'e approximants for $\mcS$ can be cast within $\left(V, \dual{\cdot}{\cdot}{V}\right)$, where $V=H^1_0(\Omega)$ and
\begin{equation*}
\dual{u}{v}{V}=\dual{\nabla u}{\nabla v}{L^2(\Omega)}+\nu^2\dual{u}{v}{L^2(\Omega)}\text{.}
\end{equation*}
We denote by $\normV{\cdot}$ the norm induced by $\dual{\cdot}{\cdot}{V}$.

The interval of interest $K$ contains two simple poles of the solution map $\lambda_1=13$ and $\lambda_2=10$, while the closest pole outside $K$ is $\lambda_3=8$. As parameters for the LS-Pad\'{e} approximant, we choose $z_0=12+\frac{i}{2}$, $\rho=R_K=\max_{z\in K}|z-z_0|$ and $N=2$, while we vary $M\in\{2,\ldots,8\}$. For the standard approximant, we set $E=M+N$, whereas we choose $E=M$ for the fast one.

To assess the accuracy of the approximation, we sample uniformly the interval $K$, and compare the numerical solution of the Helmholtz equation with the LS-Pad\'{e} approximations, measuring the error in the weighted norm $\normV{\cdot}$.

Some numerical results are shown in Figure~\ref{fig:error_std_fast}. We observe that standard and fast LS-Pad\'{e} approximants achieve a similar accuracy for a fixed numerator degree, even though the fast approximant requires the computation of $N$ fewer derivatives of the solution map. Moreover, if we compare the error that the two approximants deliver with the same amount of information (i.e. with the same $E$), we can verify that the fast LS-Pad\'{e} approximant leads to uniformly better results, which, in turn, are comparable to those obtained with a standard approximant relying on $N$ more derivatives of the solution map.
\begin{figure}[t!]
	\includegraphics[scale=1]{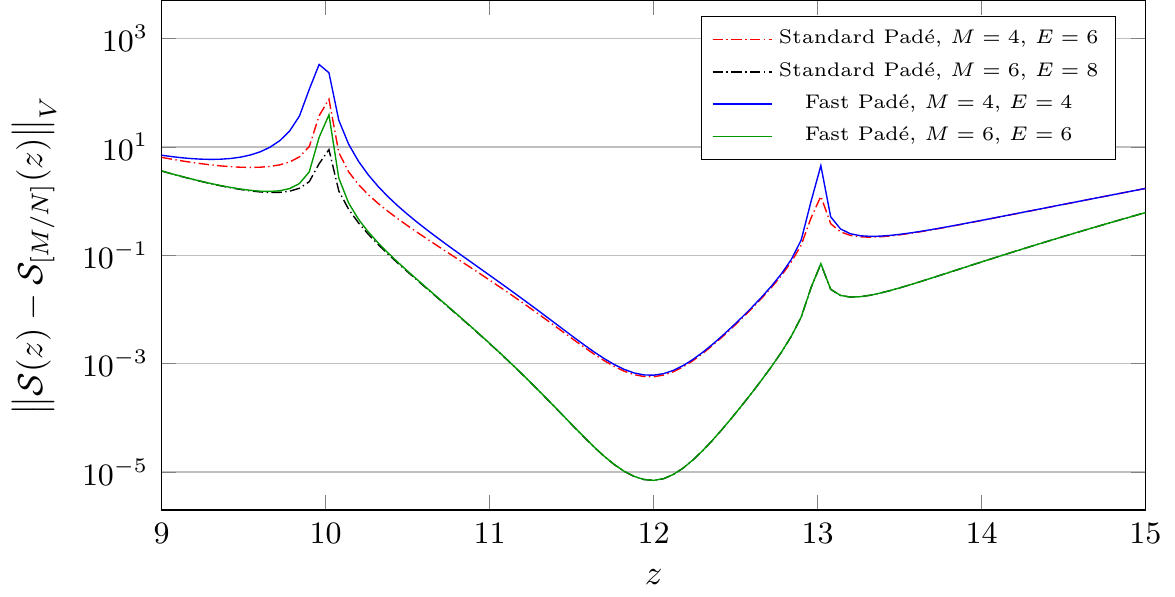}
    \caption{Error (in the weighted $H^1(D)$-norm) achieved by standard and fast Pad\'{e} approximants in the approximation of the solution map of (\ref{eq:helmholtz_parametric_weak_complex}). The high-fidelity solution (obtained with $\mathbb{P}^3$ finite elements) is computed for $n=101$ uniformly sampled values of $z\in[9,15]$.}
    \label{fig:error_std_fast}
\end{figure}
\begin{figure}[t!]
	\includegraphics[scale=1]{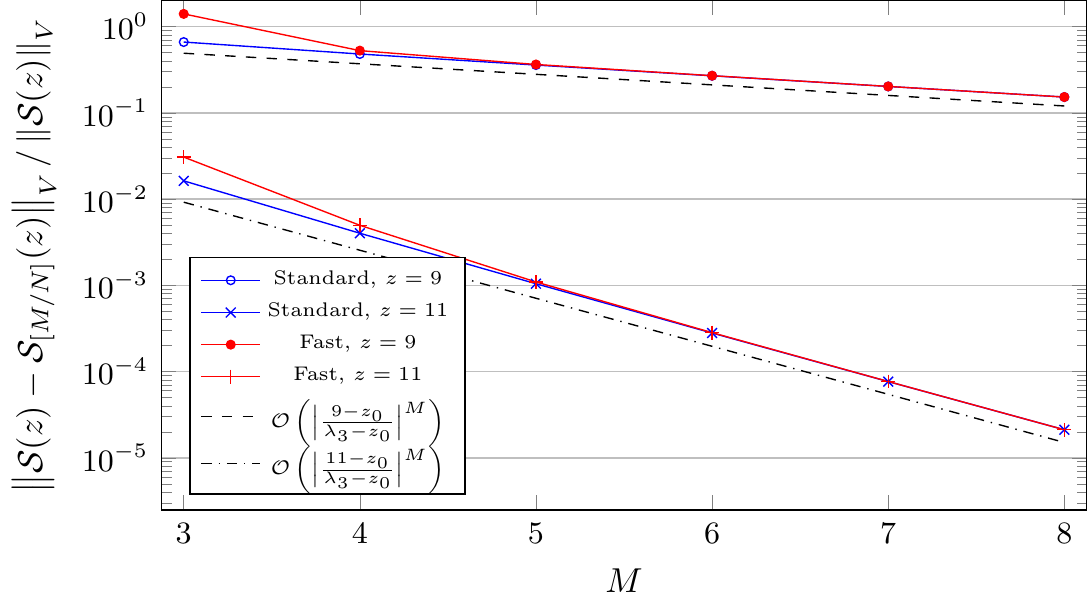}
    \caption{Convergence plots for the relative error (in the weighted $H^1(D)$-norm) achieved by standard and fast Pad\'{e} approximants at $z\in\{9,11\}$ with respect to the numerator degree. In black the convergence estimate \eqref{eq:fast_convergence_speed} for fast approximants.}
    \label{fig:error_std_fast_convergence}
\end{figure}

The error convergence in $z=9$ and $z=11$ with respect to $M$ is shown in Figure~\ref{fig:error_std_fast_convergence}. The two types of LS-Pad\'{e} approximants yield similar errors, and we can verify that the convergence rate \eqref{eq:fast_convergence_speed} holds true for both. Several numerical tests with different values of $\rho\in\{0.1R_K,R_K,10R_K\}$ have shown no evident dependence of the standard LS-Pad\'e approximation error (or of its convergence rate) on $\rho$, as \eqref{eq:pade_approx_S} could have lead to believe.

Finally, we wish to check how accurate the two LS-Pad\'{e} approximants are in the approximation of the poles of the solution map. To this aim, we compare the roots of the denominator $\den\in\PspaceN{2}{z_0}$ of each approximant with the exact poles $\lambda_1$ and $\lambda_2$. The results with respect to $E$ are shown in Figure~\ref{fig:error_std_fast_angle}. For each pole, the two types of LS-Pad\'{e} approximants seem to yield the same exponential decay. In particular, the closest pole $\lambda_1$ is approximated better than $\lambda_2$, and its error decays at a faster rate, as expected from Theorem~\ref{th:pole_convergence}, whose theoretical convergence rate \eqref{eq:pole_convergence} can be observed. Comparing the two approximation kinds, it can be observed that, for fixed $E$, the error obtained with fast LS-Pad\'{e} approximants is always smaller than the one achieved with standard approximants.
\begin{figure}[t!]
	\includegraphics[scale=1]{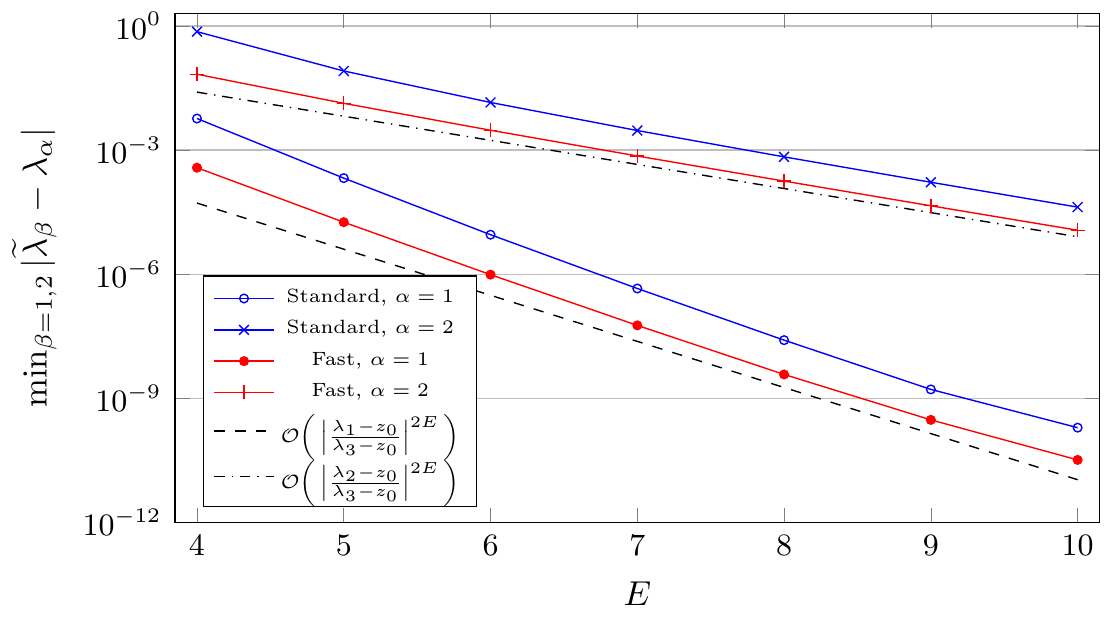}
    \caption{Convergence plot for the error in the approximation of the two closest poles of the solution map with respect to the number of computed derivatives. The results for standard LS-Pad\'{e} approximants are in blue, whereas those for fast approximants are in red. In black the \emph{a priori} convergence estimate (\ref{eq:pole_convergence}).}
    \label{fig:error_std_fast_angle}
\end{figure}

\section{Conclusions}
\label{sec:conclusions}

In this paper, we have considered Hilbert space-valued meromorphic functions arising from solution maps of parametric PDEs with the eigenproblem-like structure \eqref{eq:problem_parametric}, where $\mcL$ is an operator with normal and compact resolvent. We have proposed a rational model order reduction approach, based on single-point Least-Squares Pad\'e approximants, relying on the computation of the coefficients of the truncated Taylor series of the target function $\mcS$ at a single fixed point. The proposed approach improves, in terms of computational cost and convergence properties, the one introduced in \cite{Bonizzoni2016}, which, on the other hand, is not restricted to the case of normal operators.

Assuming the degree $N$ of the denominator of the approximant to be constant, an exponential convergence rate with respect to the number of derivatives has been proven for the error in the approximation of the target map, for values of the parameter within a disk centered at $z_0$ and encompassing $N$ poles of $\mcS$, with the exception of a set of arbitrarily small measure. A more general convergence result in measure, namely Theorem~\ref{th:fast_convergence}, has also been derived under milder conditions on the approximant type $\MNl$, including, in particular, paradiagonal approximations of type $[M/M]$ and $[M/M+1]$ with $M\to\infty$.

Moreover, it has been proven that the poles of the target function are approximated with arbitrary precision by the roots of the Pad\'e denominator, as the number of derivatives increases. In particular, an exponential convergence rate of the poles is achieved if the degree of the denominator is constant.

We believe that the description of the implementation aspects of the technique we propose has justified the word ``fast'' in the title of this work, since Krylov-based methods (in particular an Arnoldi-type algorithm, for stability purposes) can be applied to achieve a very efficient model order reduction approach.

Modifications of fast LS-Pad\'e approximants allowing snapshots of the Taylor coefficients of $\mcS$ to be taken at several points in the parameter domain are under investigation, in the spirit of rational interpolants, on the approximation theory side, and of Reduced Basis approaches, on the model order reduction side.


\begin{thebibliography}{10}

\bibitem{Ash2012}
J.~M. Ash, A. Berele, and S. Catoiu.
\newblock Plausible and Genuine Extensions of {L'H}ospital's Rule.
\newblock {\em Math. Mag.}, 85(1):52--60, 2012.

\bibitem{Babuska1997}
I.~M. Babu\v{s}ka and S.~A. Sauter.
\newblock Is the pollution effect of the {FEM} avoidable for the {H}elmholtz
  equation considering high wave numbers?
\newblock {\em SIAM J. Numer. Anal.}, 34(6):2392--2423, 1997.

\bibitem{Baker1996}
G.~A. Baker and P.~R. Graves-Morris.
\newblock {\em Pad{\'e} approximants}.
\newblock Encyclopedia Math. Appl., 1996.

\bibitem{Bhatia2007}
R. Bhatia.
\newblock {\em Perturbation Bounds for Matrix Eigenvalues}.
\newblock Classics Appl. Math., 2007.

\bibitem{Blanchard2003}
P. Blanchard and E. Br\"uning.
\newblock Spectral theory of compact operators.
\newblock {\em Mathematical methods in physics: distributions, {H}ilbert space operators, and variational methods}, Prog. Math. Phys., 327--331, 2017.

\bibitem{Bonizzoni2016}
F. Bonizzoni, F. Nobile, and I. Perugia.
\newblock Convergence analysis of {P}ad{\'e} approximations for {H}elmholtz frequency response problems.
\newblock {\em ESAIM Math. Model. Numer. Anal.}, 52(4):1261--1284, 2018.

\bibitem{Bonizzoni2018}
F. Bonizzoni, F. Nobile, I. Perugia, and D. Pradovera.
\newblock Least-{S}quares {P}ad\'e approximation of parametric and stochastic {H}elmholtz maps.
\newblock {ArXiv e-prints}, 2018. DOI: arXiv/1805.05031.

\bibitem{Chen2010}
Y. Chen, J.~S. Hesthaven, Y. Maday, and J. Rodr{\'{\i}}guez.
\newblock Certified reduced basis methods and output bounds for the harmonic {M}axwell's equations.
\newblock {\em SIAM J. Sci. Comput.}, 32(2):970--996, 2010.

\bibitem{Conway1990}
J.~B. Conway.
\newblock {\em A course in functional analysis}. Volume 96, ed. 2 of {\em Grad. Texts in Math.}.
\newblock Springer, New York, 1990.

\bibitem{Gilbarg1977}
D. Gilbarg and N.~S. Trudinger.
\newblock {\em Elliptic partial differential equations of second order}.
\newblock Classics in Math. Springer, New York, 1977.

\bibitem{Guillaume1998}
P. Guillaume, A. Huard, and V. Robin.
\newblock Generalized multivariate {P}ad{\'e} approximants.
\newblock {\em J. Approx. Theory}, 95(2):203 -- 214, 1998.

\bibitem{Hetmaniuk2012}
U. Hetmaniuk, R. Tezaur, and C. Farhat.
\newblock Review and assessment of interpolatory model order reduction methods
  for frequency response structural dynamics and acoustics problems.
\newblock {\em Internat. J. Numer. Methods Engrg.}, 90(13):1636--1662, 2012.

\bibitem{Hetmaniuk2013}
U. Hetmaniuk, R. Tezaur, and C. Farhat.
\newblock An adaptive scheme for a class of interpolatory model reduction
  methods for frequency response problems.
\newblock {\em Internat. J. Numer. Methods Engrg.}, 93(10):1109--1124, 2013.

\bibitem{Hille1965}
E.~Hille.
\newblock {\em Analytic function theory}. Ed. 2.
\newblock Chelsea Publishing Company, New York, 1965.

\bibitem{Kubrusly2012}
C.~S. Kubrusly.
\newblock {\em Spectral theory of operators on {H}ilbert spaces}.
\newblock Birkh\"auser, New York, 2012.

\bibitem{Lassila2012}
T. Lassila, A. Manzoni, and G. Rozza.
\newblock On the approximation of stability factors for general parametrized
  partial differential equations with a two-level affine decomposition.
\newblock {\em ESAIM Math. Model. Numer. Anal.}, 46(6):1555--1576, 2012.

\bibitem{Titchmarsh1978}
E.~C. Titchmarsh.
\newblock {\em The theory of functions}. Ed. 2.
\newblock Oxford Univ. Press, 1978.

\bibitem{Trefethen2009}
L.~N. Trefethen.
\newblock Householder triangularization of a quasimatrix.
\newblock {\em IMA J. Appl. Math.}, 29, 2009.

\bibitem{Veroy2003}
K. Veroy, C. Prud'Homme, D.~V. Rovas, and A.~T. Patera.
\newblock A posteriori error bounds for reduced-basis approximation of
  parametrized noncoercive and nonlinear elliptic partial differential
  equations.
\newblock {\em 16th AIAA Comp. Fluid Dyn. Conf.}, 2003.

\end{thebibliography}
\end{document}